\documentclass[reqno,12pt]{amsart}

\usepackage{enumitem}

\usepackage{vmargin}
\setpapersize{A4}
       
\usepackage{amsmath} 
\usepackage{multirow}

\usepackage{amsfonts}   
\usepackage{color}
   
\usepackage{amssymb}

\usepackage{amscd}      
 
\usepackage{amsthm}

\usepackage{tikz,amsmath}
\usetikzlibrary{arrows}

\usepackage{amstext}      
\usepackage{graphicx}

\usepackage[utf8]{inputenc} % Required for inputting international characters
\usepackage[T1]{fontenc} % Output font encoding for international characters

   \theoremstyle{plain}%default
   \newtheorem{thm}{Theorem}[section]
   \newtheorem{prop}[thm]{Proposition}
   \newtheorem{lemma}[thm]{Lemma}  
   \newtheorem{cor}[thm]{Corollary}
   \theoremstyle{defn}
   
    \newtheorem{defn}[thm]{Definition}
   \theoremstyle{remark}
   
   \newtheorem{remark}[thm]{Remark}
   \newtheorem{example}[thm]{Example}

 \newcommand{\supp}{\operatorname{supp}}

\newcommand{\G}{{\mathcal{G}}}
\newcommand{\CcG}{{C_{c}(\mathcal{G})}}

\newcommand{\Gu}{{\mathcal{G}^{(0)}}}
\newcommand{\CcGu}{{C_{c}(\mathcal{G}^{(0)})}}
\newcommand{\A}{{\mathcal{A}}}

\usepackage{lipsum}

   \numberwithin{equation}{section}

\author{Johannes Christensen}

\email{johannes@math.au.dk}
\address{Institut for Matematik, Aarhus University, Ny Munkegade, 8000 Aarhus C, Denmark}

\title{The structure of KMS weights on \'etale groupoid $C^{*}$-algebras}

\begin{document}
\maketitle

\begin{abstract}
We generalise a number of classical results from the theory of KMS states to KMS weights in the setting of $C^{*}$-dynamical systems arising from a continuous groupoid homomorphism $c:\mathcal{G} \to \mathbb{R}$ on a locally compact second countable Hausdorff \'etale groupoid $\mathcal{G}$. In particular, we generalise Neshveyev's Theorem to KMS weights.  
\end{abstract}

\bigskip
{\hspace{\parindent}  \small
}
\bigskip

\section{Introduction}
Over the past 50 years the theory of KMS weights and KMS states has come to play a significant role in operator algebras. To name a few examples, it appears as a key ingredient in Tomita-Takesaki theory \cite{Tak}, quantum groups \cite{KV} and the modeling of quantum statistical mechanical systems via $C^{*}$-algebras \cite{BR}. Furthermore, ignited by the work of Bost and Connes \cite{BC}, it has become a trend in recent years to give concrete descriptions of KMS states for different examples of $C^{*}$-dynamical systems. All this activity has resulted in a wealth of general results, in particular for KMS states on unital $C^{*}$-algebras. 

For a unital $C^{*}$-algebra all KMS weights are on the form $\lambda \omega$ with $\lambda >0$ and $\omega$ a KMS state, so studying KMS states and KMS weights is essentially the same thing. However, for non-unital $C^{*}$-algebras the two notions do not coincide. A classical textbook analogy for this, see e.g. \cite{KR}, is that considering weights instead of states corresponds to considering regular Borel measures instead of Borel probability measures on a locally compact second countable Hausdorff space. When the space is compact, corresponding to a unital $C^{*}$-algebra, the regular measures are just scalings of Borel probability measures, and hence a description of the Borel probability measures gives a description of all regular Borel measures. For non-compact spaces, corresponding to non-unital $C^{*}$-algebras, it is not in general true that all regular Borel measures are scalings of Borel probability measures, and in this case it is worth investigating the infinite measures as well. In the same manner, KMS weights are a much more appropriate invariant for non-unital $C^{*}$-dynamical systems than KMS states, see e.g. \cite{T2} for an example of this.

With this in mind, the purpose of this paper is to extend some important results from KMS states to KMS weights in the setup outlined in the abstract. In particular, we are going to
\begin{itemize}
\item embed the set of $\beta$-KMS weights into a locally convex topological vector space, and prove that the set of $\beta$-KMS weights has many of the same properties as the set of $\beta$-KMS states c.f. Theorem 5.3.30 in \cite{BR},
\item generalise a highly important theorem by Neshveyev to KMS weights, and
\item generalise a refinement of the theorem by Neshveyev presented in \cite{C2} to KMS weights.
\end{itemize}
The theorem by Neshveyev we are going to generalise is Theorem 1.3 in \cite{N}, which we henceforth will call Neshveyev's Theorem.

There are two key technical results that allow us to accomplish the above mentioned goals. The first technical result is Proposition \ref{prop31} in Section \ref{section3a}, where we prove that all KMS weights for a general $C^{*}$-dynamical system is finite on a certain subset of the Pedersen ideal. This result allows us to study the set of KMS weights in a very general setup in Section \ref{section5}. For $C^{*}$-dynamical systems arising from an \'etale groupoid $\mathcal{G}$ and a continuous groupoid homomorphism our Proposition \ref{prop31} furthermore implies that any KMS weight is finite on $C_{c}(\mathcal{G})$, which we use to generalise Neshveyev's Theorem to weights in Section \ref{section4}. It has previously been attempted to generalise Neshveyev's Theorem to weights \cite{CT5, T3}, but without the information from Proposition \ref{prop31} it has not been possible to prove the theorem in full generality. 

The second technical result is Theorem \ref{t35} in Section \ref{section3}. In Section \ref{section3} we analyse quasi-invariant regular Borel measures on \'etale groupoids, and in Theorem \ref{t35} we prove that the extremal measures in the convex set of such quasi-invariant measures are exactly the ergodic measures. This observation allows us to use ideas from ergodic theory to study KMS weights, and we use it in Section \ref{section6} to extend the main theorem of \cite{C2} to KMS weights.

%As a last remark, let us remind the reader that the terminolgy "KMS weight" on a groupoid $C^{*}$-algebra is used in two different ways in the litterature. There is the definition of a KMS weight given by Combes \cite{Combes}, which agrees with the one in e.g. \cite{K, KV}, and then there is one given by Renault, see II.5.4 in \cite{Re}, where a KMS weight is a "positive type measures" defined on $C_{c}(\mathcal{G})$. In this paper we use the definition introduced by Combes, but we will prove that these two terminologies coincides in the \'etale case.

%. This result is already known for ample \'etale groupoids \cite{CT5}, where the key idea in the proof was to use the rich supply of projections available in $C^{*}(\G)$ when $\G$ is amble. This strategy fails for general \'etale groupoids because we might not have sufficiently many projections

\smallskip

   \emph{Acknowledgement:} This work was supported by the DFF-Research Project 2 \emph{Automorphisms and Invariants of Operator Algebras}, no. 7014-00145B. Some results in this paper appeared in the authors PhD-thesis, and he thanks Jean Renault, Sergey Neshveyev and Klaus Thomsen for discussions on these.

\section{Background} \label{section2}
We will in the following summarise our notation and some basic theory on $C^{*}$-dynamical systems, groupoid $C^{*}$-algebras and KMS weights. For a more comprehensive treatment of \'etale groupoid $C^{*}$-algebras we refer to \cite{Re,Sims}, for $C^{*}$-dynamical systems we refer to \cite{K2} and for KMS weights and their GNS representation we refer to the treatments in \cite{K, KV, KV2}.

\subsection{$C^{*}$-dynamical systems}
Assume that ($\mathcal{A}$, $\alpha$) is a $C^{*}$-dynamical system over $\mathbb{R}$, i.e. $\alpha=\{\alpha_{t}\}_{t \in \mathbb{R}}$ is a continuous $1$-parameter group on the $C^{*}$-algebra $\mathcal{A}$, c.f. Definition 2.7.1 in \cite{BR}. To introduce the analytic continuations of $(\mathcal{A}, \alpha)$ we want to define a map $\alpha_{z}$ for each $z\in \mathbb{C}$ . To define its domain we let $S(z)$ denote the set of complex numbers with imaginary part between $0$ and $\text{Im}(z)$, i.e. if $\text{Im}(z)\geq 0$ then
\begin{equation*}
S(z)=\{y\in \mathbb{C} \ : \ \text{Im}(y) \in [0, \text{Im}(z)] \}
\end{equation*}
and if $\text{Im}(z)\leq 0$ then $S(z)$ is defined by requiring $\text{Im}(y) \in [\text{Im}(z), 0] $. We let $S(z)^{0}$ denote the interior of $S(z)$. The domain of $\alpha_{z}$, denoted $\text{D}(\alpha_{z})$, consist of the elements $a\in \mathcal{A}$ such that there exists a continuous function $f: S(z) \to \mathcal{A}$ that is analytic on $S(z)^{0}$ and satisfies $f(t)=\alpha_{t}(a)$ for all $t\in \mathbb{R}$. For such $a\in \text{D}(\alpha_{z})$ we define $\alpha_{z}(a)=f(z)$. We call an element $a\in \mathcal{A}$ \emph{analytic for $\alpha$} if there exists an analytic function $f: \mathbb{C} \to \mathcal{A}$ with $f(t)=\alpha_{t}(a)$ for all $t\in \mathbb{R}$, and it then follows that $a\in D(\alpha_{z})$ for each $z\in \mathbb{C}$. The operator $\alpha_{z}$ is multiplicative in the sense that when $x,y\in D(\alpha_{z})$ then $xy\in D(\alpha_{z})$ and $\alpha_{z}(xy)=\alpha_{z}(x)\alpha_{z}(y)$. 

Not all elements of $\mathcal{A}$ are analytic in general, so it is crucial that we can associate a lot of analytic elements to each $a\in \mathcal{A}$, namely for each $n\in \mathbb{N}$ the element $a(n)\in \mathcal{A}$ given by
\begin{equation} \label{esmear}
a(n) := \frac{n}{\sqrt{\pi}} \int_{\mathbb{R}} e^{-n^{2}t^{2}} \alpha_{t}(a) \ d t 
\end{equation}
is analytic, see \cite{K2} for this. These elements satisfies that $\lVert a(n)\rVert \leq \lVert a \rVert$ for all $n\in \mathbb{N}$ and that $a(n)\to a$ in norm for $n\to \infty$.

\subsection{Groupoid $C^{*}$-algebras} \label{sec22}
When $\G$ is a groupoid we define its range map $r: \mathcal{G}\to \mathcal{G}$ and its source map $s: \mathcal{G}\to \mathcal{G}$ by $r(g)=gg^{-1}$ and $s(g)=g^{-1}g$ for $g\in \mathcal{G}$. The common image of $r$ and $s$ is the \emph{unit space}, which we denote $\mathcal{G}^{(0)}$. For $x\in \mathcal{G}^{(0)}$ we set $\mathcal{G}_{x}=s^{-1}(x)$ and $\mathcal{G}^{x}=r^{-1}(x)$, and denote by $\mathcal{G}_{x}^{x}:=\mathcal{G}^{x}\cap \mathcal{G}_{x}$ the isotropy group at $x$. We will throughout the paper consider locally compact second countable Hausdorff groupoids $\mathcal{G}$ that are \'etale, meaning that $r$ and $s$ are local homeomorphisms. To ease terminology we will simply call such groupoids \'etale, so {\bf throughout this paper an \'etale groupoid is assumed to be locally compact, second countable and Hausdorff}. An open subset $W$ of an \'etale groupoid $\mathcal{G}$ will be called a bisection if both $r(W)$ and $s(W)$ are open, and $r|_{W}: W \to r(W)$ and $s|_{W}: W \to s(W)$ are homeomorphisms. Since we will work with infinite measures in this paper, it is not necessarily true that a measure restricted to a bisection $W$, or restricted to $r(W)$ or $s(W)$, is finite. Since this complicates some arguments, we will introduce the notion of a \emph{small bisection}, which will be an open bisection $W\subseteq \mathcal{G}$ with $\overline{W}$ compact such that there is another bisection $U\subseteq \mathcal{G}$ with $W \subseteq \overline{W} \subseteq U$. Since we will work with regular measures, all our measures are finite on small bisections.

For an \'etale groupoid $\G$ we can define the product $f_{1} * f_{2}$ of two functions $f_{1}, f_{2} \in C_{c}(\mathcal{G})$ via the formula
\begin{equation} \label{eprod}
(f_{1}*f_{2})(g) = \sum_{h \in \G^{r(g)}} f_{1}(h) f_{2}(h^{-1}g)
\end{equation}
for any $g\in \G$, and we can define the involution of any $f\in \CcG$ by $f^{*}(g)=\overline{f(g^{-1})}$ for $g\in \G$. With this product and involution $\CcG$ is a $*$-algebra. In this paper we will consider the completion of $\CcG$ with two different norms, the full norm $\lVert\cdot \rVert$ and the reduced norm $\lVert\cdot \rVert_{r}$, which give rise to respectively the full groupoid $C^{*}$-algebra $C^{*}(\mathcal{G})$ and the reduced groupoid $C^{*}$-algebra $C^{*}_{r}(\mathcal{G})$, see \cite{Re} for more details. The full norm on $\CcG$ is given by the formula
\begin{equation*}
\lVert f \rVert := \sup \{\lVert \pi(f) \rVert \ : \ \pi \text{ is a $* $-representation of } \CcG \}
\end{equation*}
for $f\in \CcG$, and for the definition of the reduced norm $\lVert \cdot \rVert_{r}$ we refer to \cite{Sims}, and here only remark that for $f\in \CcG$ we have the inequalities $\lVert f\rVert_{\infty} \leq \lVert f\rVert_{r} \leq  \lVert f\rVert$, where $\lVert \cdot \rVert_{\infty}$ denotes the sup-norm. When $f\in \CcG$ is supported on a bisection these inequalities become equalities, i.e. $\lVert f \rVert=\lVert f \rVert_{r}=\lVert f \rVert_{\infty}$. From this it follows that the restriction map $C_{c}(\mathcal{G})\to C_{c}(\mathcal{G}^{(0)})$ extends to a conditional expectation $P:C^{*}(\mathcal{G})\to C_{0}(\mathcal{G}^{(0)})$. When $\mathcal{H}$ is an open sub-groupoid of an \'etale groupoid $\mathcal{G}$ it follows from the definition of the full norm that the map $\iota: C_{c}(\mathcal{H})\to \CcG$ that extends functions by $0$ is an injective $*$-homomorphism which extends to a $*$-homomorphism $\iota: C^{*}(\mathcal{H})\to C^{*}(\G)$.

When $A$ is a locally compact abelian group we call a map $\Phi: \G\to A$ a \emph{continuous groupoid homomorphism} when it is continuous and $\Phi(gh)=\Phi(g)\Phi(h)$ whenever $g, h \in \G$ can be composed\footnote{These maps are also referred to as continuous $1$-cocycles in the literature}. If $A$ is a locally compact abelian group with dual group $\widehat{A}$ then a continuous groupoid homomorphism $\Phi: \G\to A$ gives rise to a $C^{*}$-dynamical system $(C^{*}(\G), \widehat{A}, \alpha)$ satisfying for $\xi \in \widehat{A}$ that
\begin{equation} \label{e1par}
\alpha_{\xi}(f)(g)=\xi(\Phi(g)) f(g) \quad \text{ for } f\in \CcG \text{ and } g\in \G
\end{equation}
see e.g. Proposition II.5.1 in \cite{Re}. When $A=\mathbb{R}$ we will denote a continuous groupoid homomorphisms by $c$, and we call the continuous $1$-parameter group it gives rise to diagonal, and denote it by $\alpha^{c}$. The $1$-parameter group $\alpha^{c}=\{\alpha^{c}_{t}\}_{t\in \mathbb{R}}$ then satisfies
\begin{equation*} 
\alpha^{c}_{t}(f)(g)=e^{itc(g)} f(g) \quad \text{ for } f\in \CcG \ , \  g\in \G  \text{ and } t\in \mathbb{R},
\end{equation*}
and the same formula defines a continuous $1$-parameter group on $C^{*}_{r}(\mathcal{G})$.

\subsection{KMS weights}
The theory of KMS weights on $C^{*}$-algebras was introduced by Combes \cite{Combes}, but in this paper we will follow the presentation in \cite{K, KV, KV2}, which are a great place for the reader to find a more rigorous treatment of the subject than the one presented here.

For a $C^{*}$-algebra $\mathcal{A}$ we let $\mathcal{A}_{+}$ denote the convex cone of positive elements in $\mathcal{A}$. A weight on $\mathcal{A}$ is a map $\psi: \mathcal{A}_{+} \to [0, \infty]$ such that $\psi(a+b)=\psi(a)+\psi(b)$ and $\psi(\lambda a)=\lambda \psi(a)$ for all $a,b \in \mathcal{A}_{+}$ and $\lambda \geq 0$. We call a weight $\psi$
\begin{itemize}
\item \emph{densely defined} if $\{ a\in \mathcal{A}_{+} \ : \ \psi(a) <\infty  \}$ is dense in $\mathcal{A}_{+}$,
\item \emph{lower semi-continuous} if $\{ a\in \mathcal{A}_{+} \ : \ \psi(a) \leq\lambda  \}$ is closed for all $\lambda \geq 0$, and
\item \emph{proper} if it is densely defined and lower semi-continuous.
\end{itemize}
For any proper weight $\psi$ we define a left ideal $\mathcal{N}_{\psi}:=\{ a\in \mathcal{A} \ :\ \psi(a^{*}a) <\infty \}$ and we set $\mathcal{M}_{\psi}^{+}:=\{a \in \mathcal{A}_{+}\ | \ \psi(a)<\infty  \}$. Setting
\begin{equation*}
  \mathcal{M}_{\psi}:=\operatorname{span}
  \{ a^{*}b \ : \ a,b \in  \mathcal{N}_{\psi} \}
\end{equation*}
then $\mathcal{M}_{\psi}=\operatorname{span} \ \mathcal{M}_{\psi}^{+}$, and it is a dense $*$-subalgebra of $\mathcal{A}$. There is a unique linear extension $\mathcal{M}_{\psi} \to \mathbb{C}$ of $\psi$ on $\mathcal{M}_{\psi} \cap \mathcal{A}_{+}$ which we also denote by $\psi$.  In this paper we will use a definition of $\beta$-KMS weights inspired by Definition 2.8 in \cite{K}. For the relation between this definition and the original one by Combes in \cite{Combes} see Theorem 6.36 in \cite{K}.

\begin{defn} \label{dweight}
Let $\mathcal{A}$ be a $C^{*}$-algebra, $\alpha : \mathbb{R} \to \text{Aut}(A)$ a continuous 1-parameter group and let $\beta \in \mathbb{R}$. We call a weight $\psi$ on $\mathcal{A}$ a $\beta$-KMS weight for $\alpha$ if it is a proper weight satisfying
\begin{enumerate}
\item $\psi \circ \alpha_{t}=\psi$ for all $t\in \mathbb{R}$.
\item \label{w2} For every $a\in D(\alpha_{-\beta i/2})$ we have 
\begin{equation*}
\psi(a^{*}a)=\psi(\alpha_{-\beta i/2}(a) \alpha_{-\beta i/2}(a)^{*} ) .
\end{equation*}
\end{enumerate}
We call $\psi$ a $\beta$-KMS state for $\alpha$ if $\sup\{\psi(a) \ : \ 0\leq a \leq 1 \}=1$.
\end{defn}

Notice that the weight $\psi=0$ is always a $\beta$-KMS weight by Definition \ref{dweight}. This differs from e.g. \cite{KV, T3} where a proper weight is by definition non-zero. We define the weight $\psi=0$ to be a $\beta$-KMS weight because we want to prove that the set of $\beta$-KMS weights for diagonal actions on \'etale groupoid $C^{*}$-algebras are closed c.f. Section \ref{section5}, which will not be true if one removes the zero weight.

If one ignore the weight $\psi=0$, the KMS weights for $\alpha$ defined in \cite{K,KV} are exactly the $-1$-KMS
weights for $\alpha$ in Definition~\ref{dweight}. For $\beta\neq 0$ we
can translate the results of \cite{K, KV} into our setting by noticing
that in Definition~\ref{dweight} a weight is a $\beta$-KMS weight for
$\{\alpha_{t}\}_{t\in \mathbb{R}}$ if and only if it is a $-1$-KMS weight for
$\{\alpha_{-\beta t}\}_{t\in \mathbb{R}}$. For $\beta=0$ we can
translate the results by using that a $0$-KMS weight for $\alpha$ in Definition \ref{dweight} is
precisely a $-1$-KMS weight for $\{\alpha_{-\beta t}\}_{t\in \mathbb{R}}=\{\text{Id}_{\mathcal{A}}\}_{t\in \mathbb{R}}$
invariant under $\alpha$.

One can extend a KMS weight on a $C^{*}$-algebra to a KMS weight on a von Neumann algebra via the GNS construction.

\begin{defn} \label{dGNS}
Let $\psi$ be a proper weight on a $C^{*}$-algebra $\mathcal{A}$. A \emph{GNS} construction for $\psi$ is a triple $(H_{\psi}, \pi_{\psi}, \Lambda_{\psi})$ where $H_{\psi}$ is a Hilbert space, $\Lambda_{\psi} : \mathcal{N}_{\psi} \to H_{\psi}$ is a linear map with dense image such that
\begin{equation*}
\langle \Lambda_{\psi}(a), \Lambda_{\psi}(b) \rangle =\psi(b^{*}a) \qquad \text{for all } a,b \in \mathcal{N}_{\psi}
\end{equation*}
and $\pi_{\psi} : \mathcal{A} \to B(H_{\psi})$ is a representation with $\pi_{\psi}(a) \Lambda_{\psi}(b)=\Lambda_{\psi}(ab)$ for all $a\in \mathcal{A}$ and $b\in \mathcal{N}_{\psi}$.
\end{defn}

As for states, one can prove that there exists a GNS
construction for a proper weight and that it is unique up to a unitary
transformation. If $\psi$ is a $\beta$-KMS weight for the $C^{*}$-dynamical system $(\mathcal{A}, \alpha)$ with GNS triple $(H_{\psi}, \pi_{\psi}, \Lambda_{\psi})$ it follows from e.g. Result 2.16 in \cite{K} that there exists a strongly continuous unitary group representation $\mathbb{R} \ni t \to u_{t}$ on $H_{\psi}$ such that $\Lambda_{\psi}(\alpha_{t}(a))=u_{t}\Lambda_{\psi}(a)$ for all $a\in \mathcal{N}_{\psi}$. Combining Result 2.3 and Lemma 4.1 in \cite{K} we see that for any $a\in \mathcal{N}_{\psi}$ then $a(n)\in \mathcal{N}_{\psi}$ for all $n\in \mathbb{N}$ and
\begin{equation}\label{eLambda}
\Lambda_{\psi}(a(n))=\frac{n}{\sqrt{\pi}}\int_{\mathbb{R}} e^{-n^{2}t^{2}} u_{t} \Lambda_{\psi}(a) \ \mathrm{d} t \to \Lambda_{\psi}(a) \text{ for } n \to \infty \ .
\end{equation}
To extend a KMS weight to a von Neumann algebra, let us
recall that a strongly continuous $1$-parameter group $\alpha$ on a
von Neumann algebra $M$ is a $1$-parameter group such that the map
$\mathbb{R} \ni t\to \alpha_{t}(a)$ is
$\sigma$-weakly continuous. We call a
weight $\phi$ on a von Neumann algebra $M$ \emph{semi-finite} when
$\{a\in M_{+} \ : \ \phi(a) <\infty\}$ is $\sigma$-weak dense in
$M_{+}$ and we call $\phi$ \emph{normal} when
$\{a\in M_{+} \ : \ \phi(a) \leq \lambda\}$ is $\sigma$-weak closed
for all $\lambda \geq 0$.

\begin{thm} \label{tgns}
Let $\psi$ be a non-zero $\beta$-KMS weight for a continuous $1$-parameter group $\alpha$ on a $C^{*}$-algebra $\mathcal{A}$ and let $(H_{\psi}, \pi_{\psi}, \Lambda_{\psi})$ be a GNS-construction for $\psi$.
\begin{enumerate}
\item There exists a normal faithful semi-finite weight $\tilde{\psi}$ on $\pi_{\psi}(\mathcal{A})''$ such that $\psi=\tilde{\psi} \circ \pi_{\psi}$.
\item $\tilde{\alpha}_{t}=\text{Ad} \ u_{t}$ defines a strongly continuous $1$-parameter group on $\pi_{\psi}(\mathcal{A})''$ with $\pi_{\psi}\circ \alpha_{t}=\tilde{\alpha}_{t}\circ\pi_{\psi}$ for all $t\in \mathbb{R}$.
\item $\{\tilde{\alpha}_{-\beta t}\}_{t\in \mathbb{R}}$ is the modular automorphism group with respect to $\tilde{\psi}$.
\end{enumerate}
\end{thm}

\begin{proof}[Proof of Theorem~\ref{tgns}]
When $\beta=-1$ these statements are the content of Section 2.2 in \cite{KV2}, and for general
  $\beta \neq 0$ the statements then follows by scaling the action. When $\beta=0$ then $(1)$ and $(3)$ follow since a $0$-KMS weight is a $-1$-KMS weight for the trivial action, and $(2)$ follows from e.g. Corollary 2.20 in \cite{KV2}.
\end{proof}

\section{KMS weights and the Pedersen ideal} \label{section3a}
In this section we will prove that for a given one-parameter group $\alpha$ on a $C^{*}$-algebra $\mathcal{A}$ there are certain positive elements in the Pedersen ideal on which any $\beta$-KMS weight for $\alpha$ will take finite values. 

Let us first recall the definition of the Pedersen ideal. Let throughout $\mathcal{A}$ be a $C^{*}$-algebra. Define the set
\begin{equation*}
F(\mathcal{A}):=\{ a\in \mathcal{A}_{+} \ | \ \exists b\in \mathcal{A}_{+} \text{ with } ab=a \}
\end{equation*}
and let
\begin{equation*}
\text{Ped}(\mathcal{A})_{+}:=\Big\{ a\in \mathcal{A}_{+} \ | \ \exists a_{1}, a_{2}, \dots, a_{n} \in F(\mathcal{A}) \text{ with } a\leq \sum_{i=1}^{n} a_{i} \Big \}.
\end{equation*}
The \emph{Pedersen ideal} $\text{Ped}(\mathcal{A})$ of $\mathcal{A}$ is then the set $\text{span}\{\text{Ped}(\mathcal{A})_{+}\}$. We refer the reader to Section 5.6 in \cite{Ped} for an introduction to the Pedersen ideal and its properties, and here we only remark that $\text{Ped}(\mathcal{A})$ is a dense ideal of $\mathcal{A}$.

\begin{prop} \label{prop31}
Let $\psi$ be a $\beta$-KMS weight for a continuous $1$-parameter group $\alpha$ on a $C^{*}$-algebra $\A$. It follows that
\begin{equation*}
 F(\mathcal{A})\cap D(\alpha_{-i\beta/2})  \subseteq \mathcal{N}_{\psi}.
\end{equation*}
\end{prop}

\begin{proof}
Let $(H_{\psi}, \pi_{\psi}, \Lambda_{\psi})$ be a GNS triple corresponding to $\psi$. Set $M=\pi_{\psi}(\A)''$ and let $\widetilde{\psi}$ be the extension of $\psi$ to $M$ c.f. Theorem \ref{tgns}. Fix $a\in F(\mathcal{A})\cap D(\alpha_{-i\beta/2})$ and pick an element $b\in \A_{+}$ with $ab=a$. We can without loss of generality assume that $\lVert a \rVert \leq 1$. The sequence $\{\pi_{\psi}(a)^{1/n}\}_{n\in \mathbb{N}}$ converges in the strong operator topology in $M$ to the range projection $p\in M$ of $\pi_{\psi}(a)$. Since $\pi_{\psi}(a)$ and $\pi_{\psi}(b)$ commute the $C^{*}$-algebra $C^{*}(\pi_{\psi}(a), \pi_{\psi}(b), 1)\subseteq M$ is commutative and the function representing $\pi_{\psi}(b)$ equals $1$ on the support of the function representing $\pi_{\psi}(a)$. Since the support of $\pi_{\psi}(a)$ and $\pi_{\psi}(a)^{1/n}$ agree for all $n$, we get that
\begin{equation*}
p \pi_{\psi}(b)=p= \pi_{\psi}(b) p.
\end{equation*}
Since $\psi$ is densely defined there exists an element $c\in \A_{+}$  with $\psi(c)<\infty$ and $\lVert c^{2}- b^{2}\rVert <1/4$. For each $n\in \mathbb{N}$ defining $c(n)$ as in \eqref{esmear} then $c(n)$ is analytic for $\alpha$ with $c(n)\in \A_{+}$ and $c(n)\to c$ for $n\to \infty$, and hence we can assume by e.g. Lemma 2.12 in \cite{K} that there is an analytic element $c \in \A_{+}$ with $\psi(c)<\infty$ and $\lVert c^{2}- b^{2}\rVert <1/2$. This gives us the following inequality
\begin{align*}
   \lVert p\pi_{\psi}(c^{2})p -p\rVert_{B(H_{\psi})}
  &=\lVert p\left(\pi_{\psi}(c^{2}) - \pi_{\psi}(b^{2})\right) p\rVert_{B(H_{\psi})} \\
  &\leq\lVert \pi_{\psi}(c^{2}) -
  \pi_{\psi}(b^{2})\rVert_{B(H_{\psi})} 
  \leq \lVert c^{2} -
  b^{2}\rVert < 1/2 .
\end{align*}
Spectral theory now implies that $p\leq 2p\pi_{\psi}(c^{2})p$. Since $p$ is the range projection of $\pi_{\psi}(a)$ then $p\pi_{\psi}(a)=\pi_{\psi}(a)$, and hence
\begin{equation*}
\pi_{\psi}(a^{2})=\pi_{\psi}(a)p\pi_{\psi}(a) \leq 2 \pi_{\psi}(a)p\pi_{\psi}(c^{2})p \pi_{\psi}(a)
=2 \pi_{\psi}(a)\pi_{\psi}(c^{2}) \pi_{\psi}(a).
\end{equation*}
This implies that
\begin{equation*}
\psi(a^{2})=\tilde{\psi}(\pi_{\psi}(a^{2}))\leq 2 \tilde{\psi}(\pi_{\psi}(a)\pi_{\psi}(c^{2}) \pi_{\psi}(a))=2\psi(ac^{2}a)
\end{equation*}
and hence to prove the proposition it suffices to argue that $\psi(ac^{2}a)<\infty$. Since $\psi$ is a $\beta$-KMS weight and $a,c\in D(\alpha_{-i\beta /2})$ this fact however follows from the calculation
\begin{align*}
\psi(ac^{2}a)&=\psi(\alpha_{-i\beta/2}(ca) \alpha_{-i\beta/2}(ca)^{*})
=\psi(\alpha_{-i\beta/2}(c) \alpha_{-i\beta/2}(a)\alpha_{-i\beta/2}(a)^{*}\alpha_{-i\beta/2}(c)^{*} ) \\
&\leq \lVert\alpha_{-i\beta/2}(a) \rVert^{2} \psi(\alpha_{-i\beta/2}(c) \alpha_{-i\beta/2}(c)^{*} ) =\lVert\alpha_{-i\beta/2}(a) \rVert^{2} \psi(c^{2} )<\infty.
\end{align*}
In conclusion $\psi(a^{2})<\infty$ and we have proved the proposition.
\end{proof}

\begin{cor} \label{cor32}
Assume that $\A$ is a $C^{*}$-algebra and $\alpha$ is a continuous one-parameter group with $F(\mathcal{A})\subseteq D(\alpha_{-i\beta /2})$. It follows that $\text{Ped}(\A)\subseteq \mathcal{M}_{\psi}$ for any $\beta$-KMS weight $\psi$.
\end{cor}
\begin{proof}
If $ab=a$ for $a,b\in\A_{+}$ then $\sqrt{a}b=\sqrt{a}$, and hence $F(\mathcal{A}) \subseteq \mathcal{M}_{\psi}^{+}$. The statement in the Corollary now follows from the definition of $\mathcal{M}_{\psi}$ and $\text{Ped}(\A)_{+}$.
\end{proof}

Remark that Corollary \ref{cor32} implies the well known result that proper tracial weights are finite on the Pedersen ideal, because any proper tracial weight can be considered a $1$-KMS weight for the trivial $1$-parameter group.

\section{The set of KMS weights}\label{section5}
The set of KMS states for a $C^{*}$-dynamical system has some remarkable properties when considered as a convex subset of the dual of the $C^{*}$-algebra, c.f. Theorem 5.3.30 in \cite{BR}. In this section we will use Proposition \ref{prop31} to analyse the structure of the set of KMS weights for a large class of $C^{*}$-dynamical systems, with the aim of proving Theorem \ref{thm46} below. The first hurdle to overcome when analysing the set of KMS weights is to find a sufficiently nice vector space to embed them into, which we do in Proposition \ref{p45} by using our result from Proposition \ref{prop31}. 

\bigskip

If $(\mathcal{A}, \alpha)$ is a $C^{*}$-dynamical system we will denote the fixed-point algebra of $\alpha$ by $\mathcal{A}^{\alpha}$, i.e. 
\begin{equation*}
\mathcal{A}^{\alpha}=\{ x\in  \mathcal{A} \ | \ \alpha_{t}(x)=x \text{ for all } t\in \mathbb{R} \} \ .
\end{equation*}
When $\mathcal{A}$ is separable the proof of Proposition 3.10.5 in \cite{Ped} shows that the following conditions are equivalent:
\begin{enumerate}
\item \label{vigtig0}There exists an approximate identity $\{E_{n}\}_{n\in \mathbb{N}}$ of $\A$ in $\A^{\alpha}$, i.e. $E_{n} \in \mathcal{A}^{\alpha}$ for all $n\in \mathbb{N}$.
\item $\A^{\alpha}$ contains a strictly positive element for $\mathcal{A}$.
\item \label{vigtig} There exists an approximate identity $\{E_{n}\}_{n\in \mathbb{N}}$ of $\A$ in $\A^{\alpha}$ with $E_{n}E_{m}=E_{n}$ whenever $m> n$.
\end{enumerate}

In this section we will analyse the structure of the set of KMS weights for a $C^{*}$-dynamical system $(\mathcal{A}, \alpha)$ with $\mathcal{A}$ separable and $\{E_{n}\}_{n\in \mathbb{N}}$ an approximate identity satisfying \eqref{vigtig}, so let us start with two examples of such $C^{*}$-dynamical systems.

\begin{example}
Let $\mathcal{G}$ be an \'etale groupoid and let $\alpha^{c}$ be a diagonal action induced by a continuous groupoid homomorphism $c:\mathcal{G}\to\mathbb{R}$ as in \eqref{e1par}. We can pick a sequence of functions $\{f_{n}\}_{n=1}^{\infty}\subseteq C_{c}(\mathcal{G}^{(0)})$ with $0\leq f_{n}\leq 1$ and $f_{n}f_{n+1}=f_{n}$ for each $n\in \mathbb{N}$ such that whenever $K\subseteq \mathcal{G}^{(0)}$ is compact there exists a $N\in \mathbb{N}$ with $f_{n}(x)=1$ for all $x\in K$ and $n\geq N$. It follows from \eqref{eprod} that this is a countable approximate identity for both $C_{r}^{*}(\mathcal{G})$ and $C^{*}(\mathcal{G})$ which satisfies \eqref{vigtig}. 
\end{example}

\begin{example}
Let $(\mathcal{A}, \alpha)$ be a $C^{*}$-dynamical system with $\mathcal{A}$ a separable $C^{*}$-algebra and $\alpha$ a periodic $1$-parameter group with period $r>0$, i.e. $\alpha_{t+r}=\alpha_{t}$ for all $t\in \mathbb{R}$. Since $\mathcal{A}$ is separable there exists a strictly positive element $x\in \mathcal{A}_{+}$. The element
\begin{equation*}
y:=\int_{0}^{r} \alpha_{t}(x) \ \mathrm{d} t
\end{equation*}
satisfies that $y\in \mathcal{A}^{\alpha}$. By definition any state $\omega$ on $\mathcal{A}$ satisfies $\omega(x)>0$. Since $\omega \circ \alpha_{t}$ is also a state for $t\in \mathbb{R}$, this implies that any state $\omega$ on $\mathcal{A}$ satisfies $\omega(y)>0$, and hence $y$ is strictly positive. In conclusion such a $C^{*}$-dynamical system satisfies \eqref{vigtig} as well.
\end{example}

To analyse the set of KMS weights fix a $C^{*}$-dynamical system $(\mathcal{A}, \alpha)$ satisfying the equivalent conditions \eqref{vigtig0}-\eqref{vigtig} with $\mathcal{A}$ separable, and fix throughout the rest of this section an approximate identity $\{E_{n}\}_{n=1}^{\infty}$ as in \eqref{vigtig}. We define $\mathcal{A}_{n}=\overline{E_{n}\mathcal{A}E_{n}}$ for all $n\in \mathbb{N}$ and denote by $\alpha^{n}$ the restriction of $\alpha$ to $\mathcal{A}_{n}$, so $(\mathcal{A}_{n}, \alpha^{n})$ is again a $C^{*}$-dynamical system. Since $E_{n+1}E_{n}=E_{n}$ we obtain an inductive system of $C^{*}$-algebras
\begin{equation} \label{edirect}
\mathcal{A}_{1} \xrightarrow{\iota_{1}} \mathcal{A}_{2} \xrightarrow{\iota_{2}} \mathcal{A}_{3} \xrightarrow{\iota_{3}} \mathcal{A}_{4} \xrightarrow{\iota_{4}}\cdots
\end{equation}
where $\iota_{n}:\mathcal{A}_{n}\to \mathcal{A}_{n+1}$ is the inclusion map, and clearly the direct limit $\varinjlim_{n\in \mathbb{N}}\mathcal{A}_{n}$ is isomorphic to $\mathcal{A}$, and $\iota_{n,\infty}:\mathcal{A}_{n}\to \mathcal{A}$ is the inclusion map. For each $n\in \mathbb{N}$ we let $\mathcal{A}^{*}_{n}$ denote the dual of $\mathcal{A}_{n}$ which is a locally convex topological vector space in the weak$^{*}$-topology. The inductive system in \eqref{edirect} gives rise to a projective system of locally convex topological vector spaces
\begin{equation*}
\mathcal{A}^{*}_{1} \xleftarrow{\pi_{1}} \mathcal{A}^{*}_{2} \xleftarrow{\pi_{2}} \mathcal{A}^{*}_{3} \xleftarrow{\pi_{3}} \mathcal{A}^{*}_{4} \cdots
\end{equation*}
where $\pi_{n}(\psi)=\psi \circ \iota_{n}$ for $\psi \in \mathcal{A}^{*}_{n+1}$. The inverse limit $\varprojlim_{n\in \mathbb{N}} \mathcal{A}_{n}^{*}$ is then a locally convex topological vector space. For each $\beta \in \mathbb{R}$ we let $\mathcal{W}^{\beta}_{n}$ denote the bounded $\beta$-KMS weights for $(\mathcal{A}_{n}, \alpha^{n})$. Since $\mathcal{W}^{\beta}_{n}$ is a closed convex subset of $\mathcal{A}^{*}_{n}$ for each $n\in  \mathbb{N}$ we can identify the inverse limit $\varprojlim_{n\in \mathbb{N}} \mathcal{W}^{\beta}_{n}$ of the projective system 
\begin{equation*}
\mathcal{W}^{\beta}_{1} \xleftarrow{\pi_{1}} \mathcal{W}^{\beta}_{2} \xleftarrow{\pi_{2}} \mathcal{W}^{\beta}_{3} \xleftarrow{\pi_{3}} \mathcal{W}^{\beta}_{4} \cdots
\end{equation*}
with a closed subset of the vector space $\varprojlim_{n\in \mathbb{N}} \mathcal{A}_{n}^{*}$. For each $\beta \in \mathbb{R}$ we let $\mathcal{W}(\beta, \alpha)$ denote the set of $\beta$-KMS weights for $\alpha$ on $\mathcal{A}$. Our goal is then to prove Proposition \ref{p45} below which identifies $\mathcal{W}(\beta, \alpha)$ with $\varprojlim_{n\in \mathbb{N}} \mathcal{W}^{\beta}_{n}$. The key technical observations which allows us to do this are contained in the following two Lemmas.

\begin{lemma}\label{lhj}
If $\omega$ is a lower semi-continuous weight on $\mathcal{A}$ with $\omega\circ \alpha_{t}=\omega$ for all $t\in \mathbb{R}$, $\omega(E_{n}^{2})<\infty$ for all $n\in \mathbb{N}$ and 
\begin{equation} \label{ehj2}
\omega(x)=\lim_{n\to \infty} \omega(E_{n} x E_{n}) \text{ for all } x\in \mathcal{A}_{+} \ ,
\end{equation}
then $\omega \in \mathcal{W}(\beta , \alpha)$ if and only if
\begin{equation} \label{ehj}
\omega(E_{n}a^{*}E_{m}^{2}aE_{n})=\omega(E_{m}\alpha_{-i\beta /2}(a)E_{n}^{2}\alpha_{-i\beta /2}(a)^{*}E_{m})
\end{equation}
for all $a\in D(\alpha_{-i\beta /2})$ and $n,m\in \mathbb{N}$. 
\end{lemma}

\begin{proof}
The ``only if'' part follows by observing that when $a\in D(\alpha_{-i\beta /2})$ then $E_{m}aE_{n}\in D(\alpha_{-i\beta /2})$ with $\alpha_{-i\beta /2}(E_{m}aE_{n})=E_{m}\alpha_{-i\beta /2}(a)E_{n}$. For the other direction assume that $\omega(E_{n}^{2}) <\infty$ for all $n\in \mathbb{N}$, that $\omega$ is lower semi-continuous, $\alpha$-invariant and satisfies \eqref{ehj2} and \eqref{ehj}. Since $\omega(E_{n}xE_{n})<\infty$ for all $n\in \mathbb{N}$ and $x\in \mathcal{A}_{+}$ we get that $\omega$ is densely defined. For any $a\in D(\alpha_{-i\beta /2})$ and $m\in \mathbb{N}$ we have by lower semi-continuity that
 \begin{equation*}
\lim_{n\to \infty}\omega(E_{m}\alpha_{-i\beta /2}(a)E_{n}^{2}\alpha_{-i\beta /2}(a)^{*}E_{m}) = \omega(E_{m}\alpha_{-i\beta /2}(a)\alpha_{-i\beta /2}(a)^{*}E_{m}) 
\end{equation*}
and hence by \eqref{ehj2} and \eqref{ehj} we have
\begin{equation*} 
\omega(a^{*}E_{m}^{2}a)=\omega(E_{m}\alpha_{-i\beta /2}(a)\alpha_{-i\beta /2}(a)^{*}E_{m}) \ .
\end{equation*}
Using the same argument for $m\to \infty$ we get that $\omega(a^{*}a)=\omega(\alpha_{-i\beta /2}(a)\alpha_{-i\beta /2}(a)^{*}) $ and hence $\omega \in \mathcal{W}(\beta , \alpha)$.
\end{proof}

\begin{lemma} \label{lem51}
Assume $\beta \in \mathbb{R}$ and assume $\{\omega_{n}\}_{n=1}^{\infty}$ is a sequence with $\omega_{n}\in \mathcal{W}^{\beta}_{n}$ for each $n\in \mathbb{N}$.
\begin{enumerate}
\item \label{lembet1}If $\omega_{n+1}|_{\mathcal{A}_{n}}\geq \omega_{n}$ for each $n\in \mathbb{N}$ and there exists a constant $K_{m}>0$ for each $m\in \mathbb{N}$ with
\begin{equation} \label{ebegr}
\omega_{n}(E_{m}^{2}) \leq K_{m} \text{  for all } n\geq m
\end{equation}
then the formula
\begin{equation} \label{evoks}
\omega(x):=\lim_{n\to \infty} \omega_{n}(E_{n} x E_{n}) \quad \text{ for } x\in \mathcal{A}_{+}
\end{equation}
defines an element $\omega \in \mathcal{W}(\beta , \alpha)$.
\item  \label{lembet2}If $\omega_{n+1}|_{\mathcal{A}_{n}}\leq \omega_{n}$ for each $n\in \mathbb{N}$ then the formula
\begin{equation} \label{eaftag}
\omega(x):=\lim_{n\to \infty} \lim_{m\to \infty} \omega_{m}(E_{n} x E_{n}) \quad x\in \mathcal{A}_{+}
\end{equation}
defines an element $\omega \in \mathcal{W}(\beta , \alpha)$.
\end{enumerate}
\end{lemma}

\begin{proof}
Let us first observe that when $\omega_{m}\in \mathcal{W}_{m}^{\beta}$ and $n<m$ then 
\begin{equation} \label{emn}
\omega_{m}(E_{n}yE_{n})\leq \omega_{m}(y) \text{ for all } y\in (\mathcal{A}_{m})_{+} \ .
\end{equation}
To see this take a $y\in \mathcal{A}_{m}$ analytic for $\alpha^{m}$. Then
\begin{align*}
\omega_{m}(E_{n}y^{*}yE_{n})&=\omega_{m}(\alpha^{m}_{-i\beta/2}(y)E_{n}^{2}\alpha^{m}_{-i\beta/2}(y)^{*}) \leq \omega_{m}(\alpha^{m}_{-i\beta/2}(y)\alpha^{m}_{-i\beta/2}(y)^{*}) \\
&=\omega_{m}(y^{*}y) 
\end{align*}
and we obtain the formula since $\omega_{m}$ is continuous.

To prove \eqref{lembet1}, fix a sequence $\{\omega_{n}\}_{n=1}^{\infty}$ satisfying the conditions in \eqref{lembet1}. We use \eqref{emn} with $m=n+1$ to get
\begin{equation*}
\omega_{n}(E_{n}xE_{n})\leq\omega_{n+1}(E_{n}xE_{n})= \omega_{n+1}(E_{n}E_{n+1}xE_{n+1}E_{n}) \leq \omega_{n+1}(E_{n+1}xE_{n+1})
\end{equation*}
for all $x\in \mathcal{A}_{+}$, and hence \eqref{evoks} defines a map $\omega: \mathcal{A}_{+} \to [0, \infty]$. It is straightforward to verify that $\omega$ is a lower semi-continuous weight with $\omega\circ \alpha_{t}=\omega$ for all $t\in \mathbb{R}$, and \eqref{ebegr} implies that $\omega(E_{n}^{2})<\infty$ for all $n\in \mathbb{N}$. Using \eqref{emn} for $y=E_{m}xE_{m}$ for some $x\in \mathcal{A}_{+}$ we get that $\omega_{m}(E_{n}xE_{n})\leq \omega_{m}(E_{m}xE_{m})$ for $n<m$, and hence $\omega(E_{n}xE_{n})\leq \omega(x)$ for all $x\in \mathcal{A}_{+}$. For each $n\in \mathbb{N}$ we have
\begin{equation*}
\omega_{n}(E_{n}xE_{n})\leq \lim_{m\to \infty}\omega_{m}(E_{n}xE_{n})=\lim_{m\to \infty}\omega_{m}(E_{m}E_{n}xE_{n}E_{m})= \omega(E_{n}xE_{n})\leq \omega(x) \ , 
\end{equation*}
and taking the limit $n\to \infty$, this inequality implies that
\begin{equation*}
\lim_{n\to \infty}\omega(E_{n}xE_{n})=\omega(x) \text{ for all }x\in \mathcal{A}_{+}.
\end{equation*}
For any $n,m\in \mathbb{N}$ and $a\in D(\alpha_{-i\beta /2})$ then
\begin{equation*}
\omega_{k}(E_{n}a^{*}E_{m}^{2}aE_{n}) = \omega_{k}(E_{m}\alpha_{-i\beta/2}(a)E_{n}^{2}\alpha_{-i\beta/2}(a)^{*}E_{m})
\end{equation*}
for all $k\geq n+m$, which implies that
\begin{equation*}
\omega(E_{n}a^{*}E_{m}^{2}aE_{n}) = \omega(E_{m}\alpha_{-i\beta/2}(a)E_{n}^{2}\alpha_{-i\beta/2}(a)^{*}E_{m}) \ ,
\end{equation*}
and hence $\omega \in \mathcal{W}(\beta , \alpha)$ by Lemma \ref{lhj}.

To prove \eqref{lembet2}, fix a sequence $\{\omega_{n}\}_{n=1}^{\infty}$ satisfying the condition in \eqref{lembet2}. Let us first argue that \eqref{eaftag} is well-defined. For each $x\in \mathcal{A}_{+}$ then $\{ \omega_{m}(E_{n} x E_{n}) \}_{m\geq n}$ is a decreasing sequence of non-negative numbers, so it converges to some number in $[0, \infty[$. For each $n\in \mathbb{N}$ and $m> n$ then using \eqref{emn} on $y=E_{n+1}xE_{n+1}$ implies that $\omega_{m}(E_{n} x E_{n})\leq \omega_{m}(E_{n+1} x E_{n+1})$, so
\begin{equation*}
\lim_{m\to \infty}\omega_{m}(E_{n} x E_{n})\leq \lim_{m\to \infty}\omega_{m}(E_{n+1} x E_{n+1}) \ ,
\end{equation*}
and it follows that the formula \eqref{eaftag} defines a weight $\omega: \mathcal{A}_{+}\to [0, \infty]$. To see that $\omega$ is lower semi-continuous choose $\{x_{k}\}_{k\in \mathbb{N}}\subseteq \A_{+}$ with $x_{k}\to x$ and $\omega(x_{k})\leq \lambda$ for all $k$ and a $\lambda >0$. If $\omega_{n}(E_{n}^{2})=0$ for arbitrarily big $n$ then $\omega=0$ and hence it is lower semi-continuous, so we can assume that there exists a $N\in \mathbb{N}$ with $\omega_{n}(E_{n}^{2})>0$ for all $n\geq N$. For each $n \geq N$ and $\varepsilon >0$ we have for $k\in \mathbb{N}$ sufficiently big that $\lVert x-x_{k}\rVert \leq \varepsilon \cdot \omega_{n}(E_{n}^{2})^{-1}$, from which it follows that
\begin{equation*}
|\omega_{m}( E_{n}(x-x_{k})E_{n})|\leq \omega_{m}(E_{n}^{2}) \frac{\varepsilon }{\omega_{n}(E_{n}^{2})}\leq \varepsilon
\end{equation*}
for $m \geq n$, and hence $\lim_{m\to \infty}\omega_{m}( E_{n}xE_{n})\leq \lambda+\varepsilon$. Since $n$ and $\varepsilon$ were arbitrary, this proves lower semi-continuity. Since 
\begin{equation}\label{e180220}
\omega(E_{k}xE_{k})=\lim_{n\to \infty} \lim_{m\to \infty} \omega_{m}(E_{n}E_{k}xE_{k}E_{n})
=\lim_{m\to \infty} \omega_{m}(E_{k}xE_{k}) 
\end{equation}
we both get that $\omega(E_{k}^{2})<\infty$ for all $k$ and that $\omega(x)=\lim_{k}\omega(E_{k}xE_{k})$ for all $x\in \mathcal{A}_{+}$. For $a\in D(\alpha_{-i\beta/2})$ we get by \eqref{e180220} that for all $l,k \in  \mathbb{N}$ then
\begin{align*}
\omega(E_{k}a^{*}E_{l}^{2}aE_{k})&=\lim_{m\to \infty} \omega_{m}(E_{k}a^{*}E_{l}^{2}aE_{k})
= \lim_{m\to \infty}\omega_{m}(E_{l}\alpha_{-i\beta/2}(a)E_{k}^{2}\alpha_{-i\beta/2}(a)^{*}E_{l}) \\
&=\omega(E_{l}\alpha_{-i\beta/2}(a)E_{k}^{2}\alpha_{-i\beta/2}(a)^{*}E_{l}) \ .
\end{align*}
Since it is straightforward to verify that $\omega \circ \alpha_{t}=\omega$, Lemma \ref{lhj} implies that $\omega$ is a $\beta$-KMS weight for $\alpha$.
\end{proof}

With Lemma \ref{lem51} we can now describe the KMS weights as a set.

\begin{prop}\label{p45}
The following map is a bijection 
\begin{equation*}
\mathcal{W}(\beta, \alpha) \ni \psi \to (\psi\circ \iota_{n, \infty})_{n=1}^{\infty} \in \varprojlim_{n\in \mathbb{N}} \mathcal{W}^{\beta}_{n} .
\end{equation*}
\end{prop}

\begin{proof}
For $\psi\in \mathcal{W}(\beta, \alpha)$ and $n\in \mathbb{N}$ then $\psi(E_{n+1}^{2})<\infty$ by Proposition \ref{prop31}, so for $y\in (\mathcal{A}_{n})_{+}$ then
\begin{equation*}
\psi(y)=\psi(E_{n+1}yE_{n+1})\leq \lVert y \rVert  \psi(E_{n+1}^{2})
\end{equation*}
which implies that the map is well defined. Assume for a contradiction that injectivity of the map fails for $\psi_{1}, \psi_{2}\in \mathcal{W}(\beta, \alpha)$, and assume that $x\in \mathcal{N}_{\psi_{1}}$ with $\psi_{2}(x^{*}x)\neq \psi_{1}(x^{*}x)$, and fix a GNS representation $(H_{i}, \pi_{i}, \Lambda_{i})$ for $\psi_{i}$ for $i=1,2$. If $\psi_{2}(x^{*}x)<\infty$ we get by Definition \ref{dGNS} and \eqref{eLambda} that
\begin{equation}\label{e20feb}
\psi_{i}(x(m)^{*}x(m))=\lVert \Lambda_{i}(x(m)) \rVert^{2}\to \lVert \Lambda_{i}(x) \rVert^{2}=\psi_{i}(x^{*}x) \text{ for } m\to \infty
\end{equation}
for $i=1,2$, and hence we can assume that there is an analytic $x$ with $\psi_{2}(x^{*}x)\neq \psi_{1}(x^{*}x)<\infty$. If $\psi_{2}(x^{*}x)=\infty$ the calculation in \eqref{e20feb} still holds for $i=1$, while by lower semi-continuity then $\psi_{2}(x(m)^{*}x(m))$ is unbounded in $m$, implying that we also can choose an analytic $x$ with $\psi_{2}(x^{*}x)\neq \psi_{1}(x^{*}x)<\infty$. So in both cases we can fix an analytic $x$ with $\psi_{2}(x^{*}x)\neq \psi_{1}(x^{*}x)$. By lower-semi continuity we have that $\lim_{n}\psi_{i}(x^{*}E_{n}^{2}x)=\psi_{i}(x^{*}x)$ for $i=1,2$. Hence there is a $n\in \mathbb{N}$ with 
\begin{equation*}
 \psi_{1}(x^{*}E_{n}^{2}x)\neq \psi_{2}(x^{*}E_{n}^{2}x)
\end{equation*}
which is a contradiction since $\psi_{i}(E_{n} \alpha_{-i\beta/2}(x)\alpha_{-i\beta/2}(x)^{*}E_{n})= \psi_{i}(x^{*}E_{n}^{2}x)$ for $i=1,2$. In conclusion $\mathcal{N}_{\psi_{1}}\subseteq \mathcal{N}_{\psi_{2}}$ and $\psi_{1}(x^{*}x)=\psi_{2}(x^{*}x)$ for $x\in \mathcal{N}_{\psi_{1}}$, and interchanging $\psi_{1}$ and $\psi_{2}$ in the above argument we get that $\psi_{1}=\psi_{2}$.

To show that the map is surjective, let $(\psi_{n})_{n\in \mathbb{N}}\in \varprojlim_{n\in \mathbb{N}} \mathcal{W}_{n}^{\beta}$. Since $\psi_{n+1}|_{\A_{n}}=\psi_{n}$ for all $n\in \mathbb{N}$ then \eqref{lembet1} in Lemma \ref{lem51} implies that setting $\psi(x):=\lim_{n\to \infty}\psi_{n}(E_{n}xE_{n})$ for each $x\in \mathcal{A}_{+}$ defines a $\beta$-KMS weights. Since $\psi$ maps to $(\psi_{n})_{n\in \mathbb{N}}$, this proves surjectivity.
\end{proof}

From the proof of surjectivity in Proposition \ref{p45} we get the following Corollary.

\begin{cor}\label{corend}
If $\omega \in \mathcal{W}(\beta, \alpha)$ then $\omega(x)=\lim_{n}\omega(E_{n}xE_{n})$ for all $x\in \mathcal{A}_{+}$.
\end{cor}

In light of Proposition \ref{p45} we will from now on identify $\mathcal{W}(\beta, \alpha)$ with $\varprojlim_{n\in \mathbb{N}} \mathcal{W}^{\beta}_{n}$, and interchangeably identify any $\beta$-KMS weight $\psi$ with its image in $\varprojlim_{n\in \mathbb{N}} \mathcal{W}^{\beta}_{n}$, which we will denote by $(\psi_{n})_{n=1}^{\infty}$.

We will now prove that $\varprojlim_{n\in \mathbb{N}} \mathcal{W}^{\beta}_{n}$ as a subset of the vector space $\varprojlim_{n\in \mathbb{N}}  \mathcal{A}_{n}^{*}$ satisfies some of the same properties that the set of $\beta$-KMS weights has on a unital $C^{*}$-algebra, c.f. Theorem 5.3.30 in \cite{BR}. For this, let us introduce the notion of an \emph{extremal KMS weight}. We say that a $\beta$-KMS weight $\psi$ in $\mathcal{W}(\beta, \alpha)$ is extremal when any $\omega, \phi \in \mathcal{W}(\beta, \alpha)$ with $\psi=\omega+\phi$ must satisfy that $\phi, \omega \in \{ \lambda \psi \ | \ \lambda \geq 0 \}$. We will now describe the set $\varprojlim_{n\in \mathbb{N}} \mathcal{W}^{\beta}_{n}$. 

\begin{thm}\label{thm46}
Consider $\varprojlim_{n\in \mathbb{N}} \mathcal{W}^{\beta}_{n}$ as a subset of the locally convex topological vector space  $\varprojlim_{n\in \mathbb{N}}  \mathcal{A}_{n}^{*}$. Then
\begin{enumerate}
\item \label{it1}$\varprojlim_{n\in \mathbb{N}} \mathcal{W}^{\beta}_{n}$ is a closed convex cone in $\varprojlim_{n\in \mathbb{N}}  \mathcal{A}_{n}^{*}$.
\item \label{it2}For $\beta \neq 0$ an element $\psi$ in $\varprojlim_{n\in \mathbb{N}} \mathcal{W}^{\beta}_{n}$ is extremal if and only if the von Neumann algebra $\pi_{\psi}(\mathcal{A})''$ is a factor.
\item \label{it3}$\varprojlim_{n\in \mathbb{N}} \mathcal{W}^{\beta}_{n}$ is a lattice in $\varprojlim_{n\in \mathbb{N}}  \mathcal{A}_{n}^{*}$ under the order
\begin{equation*}
\psi \leq \phi \iff \phi-\psi \in \varprojlim_{n\in \mathbb{N}} \mathcal{W}^{\beta}_{n}.
\end{equation*}
\end{enumerate}
\end{thm}

\begin{proof}
\eqref{it1} follows from the observation that $\mathcal{W}^{\beta}_{n}$ is a closed convex cone in $\mathcal{A}_{n}^{*}$ for each $n\in \mathbb{N}$. 

To prove the first implication in \eqref{it2} we follow the proof of Lemma 4.9 in \cite{T3}. Assume that $\psi$ is extremal in $\varprojlim_{n\in \mathbb{N}} \mathcal{W}^{\beta}_{n}$ with GNS representation $(H_{\psi}, \pi_{\psi}, \Lambda_{\psi})$, and fix a non-zero projection $p\in \pi_{\psi}(\mathcal{A})'' \cap \pi_{\psi}(\mathcal{A})'$. Let $\tilde{\psi}$ denote the extension of $\psi$ to $\pi_{\psi}(\mathcal{A})''$ and let $\tilde{\alpha}$ be the extension of $\alpha$ as in Theorem \ref{tgns}. Since $\{\tilde{\alpha}_{-\beta t}\}_{t\in \mathbb{R}}$ is the modular automorphism group for $\tilde{\psi}$ and $\beta\neq 0$ we get by e.g. Theorem 3.6 in \cite{PT} that $\tilde{\alpha}_{t}(p)=p$ for all $t\in \mathbb{R}$. It is straightforward to check that
\begin{equation*}
\mathcal{A}_{+}\ni a \to \tilde{\psi}(p\pi_{\psi}(a)) \ \text{ and } \ \mathcal{A}_{+}\ni a \to \tilde{\psi}((1-p)\pi_{\psi}(a))
\end{equation*}
define two $\beta$-KMS weights on $\mathcal{A}$ with sum $\psi$. By assumption this implies that there exists $s>0$ with $\tilde{\psi}(p \pi_{\psi}(\cdot) )=s\tilde{\psi}(\pi_{\psi}(\cdot))$. Choosing a sequence $\{a_{n}\}_{n\in \mathbb{N}}$ in $\mathcal{A}$ with $\pi_{\psi}(a_{n})\to p$ in the $\sigma$-weak operator topology, we see that for any $b\in \mathcal{N}_{\psi}$
\begin{align*}
\tilde{\psi}(p\pi_{\psi}(b)^{*}\pi_{\psi}(b)) &=\lim_{n}\tilde{\psi}(\pi_{\psi}(b)^{*}\pi_{\psi}(a_{n})\pi_{\psi}(b))=\lim_{n} \langle \Lambda_{\psi}(a_{n}b), \Lambda_{\psi}(b)\rangle \\
&=\langle p\Lambda_{\psi}(b), \Lambda_{\psi}(b)\rangle,
\end{align*}
and since $\Lambda_{\psi}(\mathcal{N}_{\psi})$ is dense in $H_{\psi}$ this implies that $s^{-1}p=I$. In conclusion $p\in \mathbb{C}I$ and $\pi_{\psi}(\mathcal{A})''$ is a factor. 

For the other direction in \eqref{it2}, assume that $\psi$ is an element of $\varprojlim_{n\in \mathbb{N}} \mathcal{W}^{\beta}_{n}$ with GNS representation $(H_{\psi}, \pi_{\psi}, \Lambda_{\psi})$ and that $\pi_{\psi}(\A)''$ is a factor. Assume that $\phi, \eta \in \varprojlim_{n\in \mathbb{N}} \mathcal{W}^{\beta}_{n}$ satisfy that $\phi+\eta=\psi$. Fix a $n\in \mathbb{N}$. Since $\mathcal{A}_{n} \subseteq \mathcal{N}_{\psi}$ we can define a closed subspace $H_{n}\subseteq H_{\psi}$ by
\begin{equation*}
H_{n}:=\overline{\Lambda_{\psi}(\A_{n})} \ ,
\end{equation*}
which becomes invariant under $\pi_{\psi}(\A_{n})$. Letting $\pi|_{n}(a)$ be the restriction of $\pi_{\psi}(a)$ to $H_{n}$ for all $a\in \mathcal{A}_{n}$, we get that $(H_{n}, \pi|_{n}, \Lambda_{\psi}|_{\mathcal{A}_{n}})$ is a GNS triple for the proper bounded weight $\psi_{n}$, which implies that $\psi_{n}$ extends to a normal bounded $\beta$-KMS weight $\tilde{\psi_{n}}$ on $\pi|_{n}(\mathcal{A}_{n})''$, c.f. Corollary 5.3.4 in \cite{BR}. By Theorem 2.3.19 in \cite{BR} $\phi_{n}$ also extends to a normal bounded $\beta$-KMS weight $\tilde{\phi_{n}}$ on $\pi|_{n}(\mathcal{A}_{n})''$ such that $\tilde{\phi_{n}} \leq \tilde{\psi_{n}}$. Since these weights are bounded, the proof of $(1)\Rightarrow (2)$ in Proposition 5.3.29 in \cite{BR} implies that there exists a unique positive operator $T_{n}\in \pi|_{n}(\mathcal{A}_{n})'' \cap \pi|_{n}(\mathcal{A}_{n})'$ of norm at most $1$ such that $\tilde{\phi_{n}}(x) = \tilde{\psi_{n}}(T_{n}x)$ for all $x\in \pi|_{n}(\mathcal{A}_{n})''$. This implies that we get a sequence of operators $\{T_{n}\}_{n=1}^{\infty}$ with $T_{n}$ an operator on $H_{n}$. For any $A, C\in \mathcal{A}_{n}$ and $B\in \mathcal{A}_{n+1}$ we have that $\pi_{\psi}(A)\pi_{\psi}(B)\Lambda_{\psi}(C)=\Lambda_{\psi}(ABC)\in H_{n}$, so $\pi_{\psi}(A)T_{n+1}\Lambda_{\psi}(C)\in H_{n}$. If $A\in (\mathcal{A}_{n})_{+}$ this implies
\begin{equation*}
T_{n+1}\Lambda_{\psi}(A)= \pi_{\psi}(\sqrt{A}) T_{n+1}\Lambda_{\psi}(\sqrt{A})\in H_{n} \ ,
\end{equation*}
and hence $T_{n+1} H_{n} \subseteq H_{n}$. Let $\Omega_{n}\in H_{n}$ be the cyclic vector with $\psi_{n}=\langle \pi_{\psi}(\cdot)\Omega_{n}, \Omega_{n} \rangle$ for each $n\in \mathbb{N}$ and let $\{B_{k}\}_{k=1}^{\infty} \subseteq \mathcal{A}_{n+1}$ satisfy that $\lim_{k} \pi|_{n+1}(B_{k})=T_{n+1}$, then for any $A\in \mathcal{A}_{n}$ we have $A^{*}B_{k}A\in \mathcal{A}_{n}$, so
\begin{align*}
\langle T_{n+1}\pi_{\psi}(A)\Omega_{n},& \pi_{\psi}(A)\Omega_{n} \rangle
=\lim_{k} \psi_{n}(A^{*}B_{k} A )=\lim_{k} \psi_{n+1}(A^{*}B_{k} A ) \\
&=\langle T_{n+1}\pi_{\psi}(A)\Omega_{n+1}, \pi_{\psi}(A)\Omega_{n+1} \rangle=\phi_{n+1}(A^{*}A)=\phi_{n}(A^{*}A) \\
&= \langle T_{n}\pi_{\psi}(A)\Omega_{n}, \pi_{\psi}(A)\Omega_{n} \rangle \ ,
\end{align*}
and hence $T_{n+1}|_{H_{n}}=T_{n}$ for all $n\in \mathbb{N}$. By Corollary \ref{corend} we get for $a\in \mathcal{N}_{\psi}$ that
\begin{equation*}
\langle \Lambda_{\psi}(a), \Lambda_{\psi}(E_{n}aE_{m}^{2}) \rangle =\psi(E_{m} a^{*} E_{n} a E_{m}) \to \psi(a^{*}a)=\lVert \Lambda_{\psi}(a)\rVert
\end{equation*}
for $n,m\to \infty$, and hence $\bigcup_{n=1}^{\infty} H_{n}$ is dense in $H_{\psi}$, and the sequence $\{T_{n}\}_{n\in \mathbb{N}}$ defines a bounded operator $T$ on $H_{\psi}$. For $A\in \pi_{\psi}(\mathcal{A}_{n})$ and $h\in H_{n}$ we have that $TAh=T_{n}Ah=ATh$, so we get that $TA=AT$, and hence $T\in \pi_{\psi}(\A)'$. Since $\Lambda_{\psi}(\mathcal{A}_{n})=\pi_{\psi}(\mathcal{A}_{n})\Lambda_{\psi}(E_{n+1})$ then $H_{n}$ is separable, so we can choose a sequence $\{\xi_{n}\}_{n=1}^{\infty}$ with $\{\xi_{n}\}_{n=1}^{N} \subseteq H_{N}$ for all $N\in \mathbb{N}$ such that $\{\xi_{n}\}_{n=1}^{\infty}$ is dense in $H_{\psi}$. For each $N\in \mathbb{N}$ we can use that $T_{N} \in \pi|_{N}(\mathcal{A}_{N})''$ to choose $A_{N} \in \mathcal{A}_{N}$ with $\lVert A_{N}\rVert \leq 2$ and
\begin{equation*}
|\langle (T_{N}-\pi_{\psi}(A_{N})) \xi_{n}  , \xi_{n}  \rangle | \leq \frac{1}{N} \ \quad \text{ for all } n \leq N \ .
\end{equation*}
It follows that $\pi_{\psi}(A_{N})\to T$, and hence $T\in \pi_{\psi}(\A)''$. By assumption this implies that $T=\lambda 1$ with $\lambda  \in \mathbb{C}$, and hence $\phi_{n}=\lambda \psi_{n}$ for all $n$, proving that $\psi$ is extremal.

To prove \eqref{it3} first notice that for $\phi=(\phi_{n})_{n=1}^{\infty}, \psi=(\psi_{n})_{n=1}^{\infty}\in \varprojlim_{n\in \mathbb{N}} \mathcal{W}^{\beta}_{n}$ we have that
\begin{equation} \label{eobsny}
\psi \leq \phi \iff \psi_{n} \leq \phi_{n} \text{ in } \mathcal{W}_{n}^{\beta} \text{ for each } n\in \mathbb{N}.
\end{equation}
Now fix two elements $\phi=(\phi_{n})_{n=1}^{\infty}$ and $ \psi=(\psi_{n})_{n=1}^{\infty}$ in $\varprojlim_{n\in \mathbb{N}} \mathcal{W}^{\beta}_{n}$. Since $\mathcal{W}^{\beta}_{n}$ is a lattice in $\mathcal{A}_{n}^{*}$ there exists a greatest lower bound $\omega_{n}\in \mathcal{W}^{\beta}_{n}$ and a least upper bound $\eta_{n}\in \mathcal{W}^{\beta}_{n}$ of the pair $\phi_{n}$ and $\psi_{n}$ for each $n\in \mathbb{N}$. Since $\omega_{n+1}|_{\mathcal{A}_{n}}$ is a lower bound of $\phi_{n}$ and $\psi_{n}$ we have $\omega_{n+1}|_{\mathcal{A}_{n}} \leq \omega_{n}$, and likewise $\eta_{n+1}|_{\mathcal{A}_{n}} \geq \eta_{n}$. Using that $\eta_{m}|_{\mathcal{A}_{n}}\leq (\psi+\phi)|_{\mathcal{A}_{n}}$ for all $m\geq n$ it follows from Lemma \ref{lem51} that the formulas
\begin{equation*}
\sigma(x):=\lim_{n\to \infty} \eta_{n}(E_{n} x E_{n}) \text{ for } x\in \mathcal{A}_{+}
\end{equation*}
and
\begin{equation*}
\tau(x):=\lim_{n\to \infty} \lim_{m\to \infty}\omega_{m}(E_{n} x E_{n}) \text{ for } x\in \mathcal{A}_{+}
\end{equation*}
define two $\beta$-KMS weights. We claim that $\sigma $ is a least upper bound of $\psi$ and $\phi$ and $\tau $ is a greatest lower bound of $\psi$ and $\phi$.

To prove that $\tau$ is a greatest lower bound notice that $\tau \circ \iota_{n, \infty}=\lim_{m\to \infty} \omega_{m}|_{\mathcal{A}_{n}}$, so since $\mathcal{W}_{n}^{\beta}$ is closed we get that $\tau\circ \iota_{n, \infty}$ is a lower bound of $\psi_{n}$ and $\phi_{n}$ for all $n\in \mathbb{N}$.

If $\tau'$ is another lower bound of $\psi$ and $\phi$, then $\tau'\circ \iota_{m, \infty}\leq \omega_{m}$ for all $m$, so since 
\begin{equation*}
\tau' \circ \iota_{n, \infty} =( \tau' \circ  \iota_{m, \infty})|_{\A_{n}}
 \leq \omega_{m}|_ {\A_{n}}
\end{equation*}
for all $m\geq n$, we see that $\tau'\leq \tau$, proving the $\tau$ is the greatest lower bound. A similar argument proves that $\sigma$ is a least upper bound, proving \eqref{it3}.
\end{proof}

\section{Quasi-invariant measures}\label{section3}
For an \'etale groupoid $\mathcal{G}$, a real number $\beta \in \mathbb{R}$, and a continuous groupoid homomorphism $c: \mathcal{G} \to \mathbb{R}$ the quasi-invariant measures on $\mathcal{G}^{(0)}$ with Radon-Nikodym cocycle $e^{-\beta c}$ play a crucial role in the description of $\beta$-KMS states.  This was first observed by Renault in his pionering thesis, c.f. Proposition II.5.4 in \cite{Re}, and later generalised in Theorem 1.3 in \cite{N}. The aim of this section is therefore to analyse these measures, but since all results are valid for general quasi-invariant measures we will state and prove them in this generality. Most of the results in this section are known to experts when the measure is a probability measure, but since we are dealing with general regular measures, and since there seems to be no reference for most of these facts, we will include proofs. The main result of the section, Theorem \ref{t35}, is new even for probability measures.

\begin{lemma} \label{l31}
When $\mathcal{G}$ is an \'etale groupoid there exists a countable basis for the topology on $\mathcal{G}$ consisting of small bisections, and $r$ and $s$ maps Borel sets to Borel sets.
\end{lemma}
\begin{proof}
The first fact follows since $\mathcal{G}$ is second countable and since an open subset of a bisection is a bisection. If $\{W_{i}\}_{i\in \mathbb{N}}$ is a countable basis of small bisections and $B$ is Borel, we have
\begin{equation*}
r(B)=\bigcup_{i=1}^{\infty} r(B\cap W_{i}) .
\end{equation*}
Now $r(B\cap W_{i})$ is Borel since $r|_{W_{i}}:W_{i}\to r(W_{i})$ is a homeomorphism, so $r(B)$ is Borel. The statement concerning $s$ follows similarly.
\end{proof}

Let $\mathcal{G}$ be an \'etale groupoid, and let $\mu$ be a regular Borel measure on $\mathcal{G}^{(0)}$. Using Riesz representation theorem we obtain two unique regular Borel measures $\mu_{r}$ and $\mu_{s}$ on $\mathcal{G}$ such that
\begin{equation} \label{eny1}
\int_{\mathcal{G}} f\ d\mu_{r} = \int_{\mathcal{G}^{(0)}} \sum_{g\in \mathcal{G}^{x}} f(g) \ d\mu(x) \ , \ \int_{\mathcal{G}} f\ d\mu_{s} = \int_{\mathcal{G}^{(0)}} \sum_{g\in \mathcal{G}_{x}} f(g) \ d\mu(x) \
\end{equation}
for all $f\in C_{c}(\mathcal{G})$.

\begin{defn}
Let $\mathcal{G}$ be an \'etale groupoid and let $\mu$ be a regular Borel measure on $\mathcal{G}^{(0)}$. We call $\mu$ quasi-invariant with Radon-Nikodym cocycle $\kappa$ if $\mu_{r}$ and $\mu_{s}$ are equivalent and $d \mu_{r} / d\mu_{s}=\kappa$ for a Borel function $\kappa:\mathcal{G} \to \mathbb{R}$.
\end{defn}

If $\mu_{r}$ and $\mu_{s}$ are equivalent and $d \mu_{r} / d\mu_{s}=\kappa$ we can and will assume that $\kappa$ is positive everywhere, and it then follows that $d \mu_{s} / d\mu_{r}=\kappa^{-1}$.

\begin{prop} \label{p33}
Let $\mathcal{G}$ be an \'etale groupoid and let $\mu$ be a regular Borel measure on $\mathcal{G}^{(0)}$. The following are equivalent:
\begin{enumerate}
\item \label{1quasi}$\mu$ is quasi-invariant with Radon-Nikodym cocycle $\kappa$.
\item \label{2quasi}For all small bisections $W\subseteq \mathcal{G}$ we have
\begin{equation}\label{e3p33}
\mu(s(W))=\int_{r(W)}\kappa(r_{W}^{-1}(x))^{-1} \ d\mu (x)
\end{equation}
where $r_{W}^{-1}$ is the inverse of $r_{W}:W \to r(W)$.
\item \label{2quasia} Equation \eqref{e3p33} is true for all bisections $W\subseteq \G$. 
\item \label{3quasi} Whenever $B\subseteq W$ is Borel for some small bisection $W$
\begin{equation*}
\mu(s(B))=\int_{r(B)} \kappa(r_{W}^{-1}(x))^{-1} \ d\mu (x) .
\end{equation*}
\end{enumerate}
\end{prop}

\begin{proof}
Let $W$ be a small bisection and let $\mu$ be a regular Borel measure on $\mathcal{G}^{(0)}$. For $h\in C_{c}(W)$ then \eqref{eny1} implies that
\begin{equation*}
\int_{\mathcal{G}} h\ d\mu_{r} = \int_{r(W)} h(r_{W}^{-1}(x)) \ d\mu(x) \ \text{ and } \ \int_{\mathcal{G}} h\ d\mu_{s} = \int_{s(W)}h(s_{W}^{-1}(x)) \ d\mu(x) .
\end{equation*}
This guarantees that
\begin{equation}\label{eny3}
\mu_{s}(B)=\mu(s(B)) \text{ and } \mu_{r}(B)=\mu(r(B)) \text{ for all Borel } B\subseteq W \ .
\end{equation}

To see that (\ref{1quasi}) implies (\ref{2quasi}) assume that $\mu$ is quasi-invariant with Radon-Nikodym cocycle $\kappa$ and let $W\subseteq \mathcal{G}$ be a small bisection. Since $d\mu_{s} / d\mu_{r} = \kappa^{-1}$ we get
\begin{equation*}
\mu(s(W))=\mu_{s}(W)=\int_{W} \kappa^{-1} d\mu_{r}= \int_{r(W)} \kappa(r_{W}^{-1}(x))^{-1} \ d\mu (x)
\end{equation*}
by using \eqref{eny3}. The equivalence of (\ref{2quasi}) and (\ref{2quasia}) follows by writing any bisection $W$ as the union of an increasing sequence of small bisections. For (\ref{2quasi}) implies (\ref{3quasi}) let $W$ be a small bisection and let $\mathcal{B}(W)$ denote the Borel subsets of $W$, and fix $B'\in  \mathcal{B}(W)$. The two finite Borel measures on $W$ given by
\begin{equation*}
\mathcal{B}(W) \ni B \to \mu(s(B)) \quad , \quad \mathcal{B}(W) \ni B \to \int_{r(B)} \kappa(r_{W}^{-1}(x))^{-1} \ d\mu (x)
\end{equation*}
agree on open sets by assumption, so by regularity they agree on $B'$, which proves (\ref{2quasi}) implies (\ref{3quasi}). To see that (\ref{3quasi}) implies (\ref{1quasi}), notice that we can write any Borel $B\subseteq \mathcal{G}$ as a countable disjoint union
$\bigsqcup_{i} B_{i}$ with each $B_{i}$ Borel and contained in a small
bisection. Fix $i$ and suppose $B_{i} \subseteq W$ for a small
bisection $W$, it follows from \eqref{eny3} that
\begin{equation*}
\mu_{s}(B_{i})=\mu(s(B_{i}))= \int_{r(B_{i})} \kappa(r_{W}^{-1}(x))^{-1} \ d\mu (x) = \int_{B_{i}} \kappa(g)^{-1} d\mu_{r}(g)  .
\end{equation*}
This proves that (\ref{3quasi}) implies (\ref{1quasi}).
\end{proof}

\begin{lemma} \label{l34}
Let $\mathcal{G}$ be an \'etale groupoid. For $N\subseteq \Gu$ then
\begin{equation*}
s(r^{-1}(N))=r(s^{-1}(N))
\end{equation*}
and this set is Borel if $N\subseteq \Gu$ is Borel. If $\mu$ is quasi-invariant with Radon-Nikodym cocycle $\kappa$ and $\mu(N)=0$ for a Borel set $N$ then $\mu(s(r^{-1}(N)))=0$.
\end{lemma}

\begin{proof}
The first statement follows by using that $r(g)=s(g^{-1})$ for $g\in \G$, and the second follows since $s$ is continuous and $r$ maps Borel sets to Borel sets by Lemma \ref{l31}. If $W$ is a small bisection and $\mu(N)=0$ then Proposition~\ref{p33} implies
\begin{equation*}
\mu(s(r_{W}^{-1}(N \cap r(W))))= \int_{N \cap r(W)} \kappa(r_{W}^{-1}(x))^{-1} \ d\mu(x) =0.
\end{equation*}
Taking a basis $\{W_{i}\}_{i=1}^{\infty}$ of small
  bisections then
\begin{equation*}
s(r^{-1}(N))=\bigcup_{i=1}^{\infty} s(r_{W_{i}}^{-1}(N \cap r(W_{i})))
\end{equation*}
which proves $\mu(s(r^{-1}(N)))=0$.
\end{proof}

The main result in this section, which will be a key tool in several arguments in this paper, is Theorem \ref{t35} below. In spirit it is closely related to the classical observation in ergodic theory, that the extremal invariant probability measures for a homeomorphism on a compact metric space are exactly the ergodic invariant probability measures.

For an \'etale groupoid $\G$ and a positive Borel map $\kappa: \G\to \mathbb{R}$ we let $\Delta(\kappa)$ denote the set of regular Borel measures on $\Gu$ that are quasi-invariant with Radon-Nikodym cocycle $\kappa$. We say $\mu \in \Delta(\kappa)$ is \emph{extremal} when any $\mu_{1}, \mu_{2} \in\Delta(\kappa)$ with $\mu=\mu_{1}+\mu_{2}$ satisfies $\mu_{1}, \mu_{2} \in \{ \lambda \mu  :  \lambda \geq 0 \}$. Following \cite{Re} we call a set $B \subseteq \mathcal{G}^{(0)}$ invariant if $
B=r(s^{-1}(B))$. A Borel measure $\mu$ on $\mathcal{G}^{(0)}$ is then called \emph{ergodic} if for all invariant Borel sets $B$ we either have $\mu(B) =0$ or $\mu(B^{C})=0$. 

\begin{thm} \label{t35}
Let $\G$ be an \'etale groupoid and $\kappa:\G\to \mathbb{R}$ a positive Borel map. A measure in $\Delta(\kappa)$ is extremal if and only if it is ergodic.
\end{thm}

\begin{proof}
If $B$ is invariant then $B^{C}$ is invariant, so to prove that extremal measures are ergodic it suffices to prove that if $B\subseteq \Gu$ is Borel and invariant and $\mu \in \Delta(\kappa)$ then $\mu_{B}(\cdot )=\mu(B\cap \cdot\ ) \in \Delta(\kappa)$. Let $W$ be a small bisection in $\mathcal{G}$, then
\begin{equation*}
s(W) \cap B = s(r_{W}^{-1}(r(W) \cap B)) .
\end{equation*}
Using Proposition~\ref{p33} we get that
\begin{equation*}
\mu_{B}(s(W)) = \int_{r(W) \cap B} \kappa(r_{W}^{-1}(x))^{-1} \ d\mu (x) = \int_{r(W) } \kappa(r_{W}^{-1}(x))^{-1} \ d\mu_{B}(x)
\end{equation*}
and hence $\mu_{B}\in \Delta(\kappa)$.

Assume now that $\mu$ is ergodic, and that $\mu_{1}, \mu_{2} \in \Delta(\kappa)\setminus \{0\}$ satisfy that $\mu=\mu_{1}+\mu_{2}$. Since $\mathcal{G}^{(0)}$ is $\sigma$-compact the Radon-Nikodym theorem implies that there exists non-negative Borel functions $f_{i}:\mathcal{G}^{(0)} \to [0, \infty[$ for $i=1,2$ such that
\begin{equation*}
\mu_{i}(B)=\int_{B} f_{i}(x) \ d \mu(x)
\end{equation*} 
for all Borel sets $B \subseteq \mathcal{G}^{(0)}$. Since $\mu=\mu_{1}+\mu_{2}$ we have that $f_{1}(x)+f_{2}(x)=1$ for $\mu$-a.e. $x$. If $f_{1}$ is constant $\mu$-almost everywhere it follows that $\mu_{1}$ and $\mu_{2}$ are scalings of $\mu$, completing the proof, so assume for a contradiction this is not the case. It follows that we can find a $t\in ]0,1[$ such that $\mu(f_{1}^{-1}([0,t[))>0$ and $\mu(f_{1}^{-1}(]t,1]))>0$.

Lemma~\ref{l34} implies that
$B:=s(r^{-1}(f_{1}^{-1}([0,t[)))$ is Borel and invariant, and since $\mu(B)\geq \mu(f_{1}^{-1}([0,t[))>0$ we must have that $\mu(B^{C})=0$ by ergodicity of $\mu$. Since $\mu(f_{1}^{-1}(]t,1]))>0$ and
$f_{1}(x)+f_{2}(x)=1$ for $\mu$-a.e. $x$ we have that
\begin{equation*}
\mu(f_{2}^{-1}([0,1-t[))>0,
\end{equation*}
so setting $C:=s(r^{-1}(f_{2}^{-1}([0,1-t[)))$ we likewise get $\mu(C^{C})=0$.

\bigskip

\emph{Claim:} $\mu(\{x \in B \ |\ f_{1}(x)>t \})=0$ and $\mu(\{x \in C \ |\ f_{2}(x)>1-t \})=0$.

\bigskip

If this claim is true we reach the contradiction as follows: By construction $\mu(B^{C} \cup C^{C})=0$, and hence for $\mu$-a.e. $x $ we have $f_{1}(x)\leq t$ and $f_{2}(x) \leq 1-t$ by the claim. Since $f_{1}(x)+f_{2}(x)=1$ for $\mu$-a.e. $x$ this implies that $f_{1}(x)=t$ for $\mu$-a.e. $x$, contradicting that $f_{1}$ was not constant. In conclusion $\mu_{1}, \mu_{2} \in \mathbb{R}_{+} \mu$, proving that $\mu$ is extremal.

To prove the claim, assume for a contradiction that
\begin{equation*}
\mu(\{x \in B \ |\ f_{1}(x)>t \})>0.
\end{equation*} 
Let $\{W_{j}\}_{j=1}^{\infty}$ be a countable basis of small bisections, then
\begin{equation*}
B=s(r^{-1}(f_{1}^{-1}([0,t[))) = \bigcup_{j=1}^{\infty} s\bigg(r_{W_{j}}^{-1}\Big(f_{1}^{-1}([0,t[) \cap r(W_{j})\Big)\bigg).
\end{equation*} 
Since the $W_{j}$'s are small bisections, there exists a $j$ such that $0<\mu(H)<\infty$ where
\begin{equation*}
H:=\left\{x \in s\big (r_{W_{j}}^{-1}\left[f_{1}^{-1}([0,t[) \cap r(W_{j})\right] \big) \ : \  f_{1}(x) >t\right\}.
\end{equation*}
By definition of $f_{1}$ we get that
\begin{equation*} 
\mu_{1}(H)=\int_{H} f_{1}(x) \ d \mu(x) >t \mu(H)
\end{equation*}
but on the other hand, since $r(s_{W_{j}}^{-1}(H))\subseteq f_{1}^{-1}([0,t[)$ we get
\begin{align*}
\mu_{1}(H) &= \mu_{1}\left(s\big(s_{W_{j}}^{-1}(H)\big)\right)
=\int_{r(s_{W_{j}}^{-1}(H))} \kappa(r_{W_{j}}^{-1}(x))^{-1} \ d \mu_{1}(x) \\
&\leq t \int_{r(s_{W_{j}}^{-1}(H))} \kappa(r_{W_{j}}^{-1}(x))^{-1} \ d \mu(x)
=t \mu\left(s\big(s_{W_{j}}^{-1}(H)\big)\right)=t \mu(H)
\end{align*}
a contradiction. It follows that $\mu(\{x \in B \ |\ f_{1}(x)>t \})=0$. We leave it to the reader to check that the proof that $\mu(\{x \in C \ |\ f_{2}(x)>1-t \})=0$ follows in exactly the same way. In conclusion the claim is true, which proves that ergodic measures are extremal.
\end{proof}

\begin{remark}\label{lergokonstant}
One of the key ideas in this paper is to use ideas from ergodic theory to analyse KMS weights, which is possible because of Theorem \ref{t35}. In particular we will use that when $\mu\in \Delta(\kappa)$ is ergodic and $X$ is a second countable metric space, then any Borel function $f:\Gu \to X$ that is constant on the sets $s(r^{-1}(\{x\}))$ for $\mu$-a.e. $x\in \Gu$ is constant $\mu$-almost everywhere.
\end{remark}

\section{Neshveyev's Theorem}\label{section4}
In this section we will generalise Neshveyev's Theorem \cite{N} from KMS states to KMS weights, and we will use it to answer some open questions in the literature. The first step in this analysis is to prove that any KMS weight $\psi$ for a diagonal action $\alpha^{c}$ is finite on $C_{c}(\mathcal{G})$, which follows from Proposition \ref{prop31}. 

\begin{prop} \label{p41}
Let $\mathcal{G}$ be an \'etale groupoid, let
  $c : \mathcal{G} \to \mathbb{R}$ be a continuous groupoid
  homomorphism and let $\beta \in \mathbb{R}$. If $\psi$ is a $\beta$-KMS
  weight on $C^{*}(\mathcal{G})$ for $\alpha^{c}$ then $\CcG \subseteq \mathcal{M}_{\psi}$. 
\end{prop}

\begin{proof}
If $f\in C_{c}(\Gu)_{+}$ then by definition of the action $\alpha^{c}$ we have that $\sqrt{f} \in C_{c}(\Gu)_{+} \cap D(\alpha^{c}_{-i\beta /2})$. Since the product of functions in $C_{c}(\Gu)$ is pointwise, we have that $C_{c}(\Gu)_{+} \subseteq F(C^{*}(\mathcal{G}))$ and hence $\psi(f)<\infty$ by Proposition \ref{prop31}. For a general function $g\in \CcG$ we can find a positive function $f\in \CcGu_{+}$ such that $g=fgf$. Writing $g=\sum_{i=1}^{4}\lambda_{i} b_{i}$ with $b_{i}\in C^{*}(\G)_{+}$ then since $\psi(fb_{i}f)\leq \lVert b_{i} \rVert \psi(f^{2})<\infty$ we get that
\begin{equation*}
g=fgf=\sum_{i=1}^{4}\lambda_{i} fb_{i}f \in \mathcal{M}_{\psi}
\end{equation*}
which proves the proposition. 
\end{proof}

To extend Neshveyev's theorem to weights let us recall the definiton of \emph{ $\mu$-measurable field of states} from \cite{N}.

\begin{defn}
Let $\mathcal{G}$ be an \'etale groupoid, and let $\mu$ be a regular Borel measure on $\mathcal{G}^{(0)}$. For each $x\in \mathcal{G}^{(0)}$ we let $u_{g}$, $g\in \mathcal{G}_{x}^{x}$ denote the canonical unitary generators of $C^{*}(\mathcal{G}_{x}^{x})$. We call a collection $\{\varphi_{x}\}_{x\in \mathcal{G}^{(0)}}$ a \emph{$\mu$-measurable field of states} if each $\varphi_{x}$ is a state on $C^{*}(\mathcal{G}_{x}^{x})$ and the function
\begin{equation} \label{emumeasure}
\mathcal{G}^{(0)} \ni x \to \sum_{g\in \mathcal{G}_{x}^{x}} f(g) \varphi_{x}(u_{g})
\end{equation}
is $\mu$-measurable for each $f\in C_{c}(\mathcal{G})$.

We identify two $\mu$-measurable fields $\{\varphi_{x}\}_{x\in \mathcal{G}^{(0)}}$ and $\{\varphi_{x}'\}_{x\in \mathcal{G}^{(0)}}$ when $\varphi_{x}'=\varphi_{x}$ for $\mu$-a.e. $x\in \mathcal{G}^{(0)}$.
\end{defn}

We are now ready to state and prove Neshveyev's Theorem for weights. This has already been done for \'etale groupoids $\mathcal{G}$ with $\mathcal{G}^{(0)}$ totally disconnected in Theorem 3.2 in \cite{CT5}, and many of the ideas in the proof of Theorem \ref{tNW} below are inspired by the proof of Theorem 3.2 in \cite{CT5}. Since a lot of the technical details are significantly different we will however write out the proof.

\begin{thm}[Neshveyev's Theorem for Weights] \label{tNW}
Let $\mathcal{G}$ be an \'etale groupoid, let $c: \mathcal{G} \to \mathbb{R}$ be a continuous groupoid homomorphism and let $\beta \in \mathbb{R}$.

There is a bijective correspondence between the $\beta$-KMS weights for $\alpha^{c}$ on $C^{*}(\mathcal{G})$ and the pairs $(\mu, \{ \varphi_{x}\}_{x\in \mathcal{G}^{(0)}})$, where $\mu$ is a regular Borel measure on $\mathcal{G}^{(0)}$ and $\{ \varphi_{x}\}_{x\in \mathcal{G}^{(0)}}$ is a $\mu$-measurable field of states $\varphi_{x}$ on $C^{*}(\mathcal{G}_{x}^{x})$ such that
\begin{enumerate}
\item \label{eb1}$\mu$ is quasi-invariant with Radon-Nikodym cocycle $e^{-\beta c}$.
\item \label{eb2} $\varphi_{x}(u_{g})=\varphi_{r(h)}(u_{hgh^{-1}})$ for $\mu$-a.e. $x\in \mathcal{G}^{(0)}$ and all $g\in \mathcal{G}_{x}^{x}$ and $h\in \mathcal{G}_{x}$.
\item \label{eb3}$\varphi_{x}(u_{g})=0$ for $\mu$-a.e. $x\in \mathcal{G}^{(0)}$ and all $g\in \mathcal{G}_{x}^{x}\setminus c^{-1}(0)$.
\end{enumerate}
The $\beta$-KMS weight $\psi$ corresponding to the pair $(\mu, \{ \varphi_{x}\}_{x\in \mathcal{G}^{(0)}})$ has the property that $C_{c}(\mathcal{G})\subseteq \mathcal{M}_{\psi}$ and it is the unique $\beta$-KMS weight satisfying
\begin{equation}\label{eCcG}
\psi(f) = \int_{\mathcal{G}^{(0)}} \sum_{g\in \mathcal{G}_{x}^{x}} f(g) \varphi_{x}(u_{g}) \ d \mu(x)
\end{equation} 
for all $f\in C_{c}(\mathcal{G})$.
\end{thm}

\begin{proof}
Since this proof is quite long we will divided it into four steps. Throughout $\{V_{i}\}_{i=1}^{\infty}$ denotes a sequence of open
sets in $\mathcal{G}^{(0)}$ with compact closure $\overline{V_{i}} \subseteq V_{i+1}$ for
each $i$ such that
\begin{equation*}
\mathcal{G}^{(0)}= \bigcup_{i=1}^{\infty} V_{i} \ . 
\end{equation*}
We choose a sequence of functions
$\{E_{n} \}_{n=1}^{\infty} \subseteq C_{c}(\mathcal{G}^{(0)})$ with
the property that $E_{n}(x)=1$ for $x\in \overline{V_{n}}$,
$\supp(E_{n}) \subseteq V_{n+1}$ and with
$0\leq E_{n}(x) \leq 1$ for all $x\in \mathcal{G}^{(0)}$. Then $\{E_{n}\}_{n\in \mathbb{N}}$ is an
approximate identity on $C^{*}(\mathcal{G})$ contained in $C^{*}(\mathcal{G})^{\alpha^{c}}$. For each $n\in \mathbb{N}$ we let $\G_{n}$ denote the open subgroupoid
\begin{equation*}
\mathcal{G}_{n}:=\mathcal{G}|_{V_{n}}=\{g\in \G \ : \ r(g),s(g)\in V_{n}\} \ ,
\end{equation*}
and we let $\iota_{n}$ denote the $*$-homomorphism $\iota_{n}: C^{*}(\G_{n})\to C^{*}(\G)$ described in Section \ref{sec22}.\newline

\paragraph{\textbf{ Step $1$: Every pair $(\mu, \{\varphi_{x}\}_{x\in \mathcal{G}^{(0)}})$ gives rise to a $\beta$-KMS weight}} \hfill

\noindent Assume $(\mu, \{\varphi_{x}\}_{x\in \mathcal{G}^{(0)}})$ satisfies \eqref{eb1} -- \eqref{eb3}. If $\mu=0$ we associate the $\beta$-KMS weight $\psi=0$ to the pair, so assume that $\mu \neq0$. Then every $x\in \mathcal{G}^{(0)}$ gives rise to a state $\psi_{x}$ on $C^{*}(\mathcal{G})$ such that
\begin{equation*}
\psi_{x}(f) =\sum_{g\in \mathcal{G}_{x}^{x}} f(g) \varphi_{x}(u_{g})
\end{equation*}
for all $f\in C_{c}(\mathcal{G})$, c.f. Theorem 1.1 in \cite{N}. Since $x\to \psi_{x}(f)$ is $\mu$-measurable for $f\in C_{c}(\mathcal{G})$, so is $x\to \psi_{x}(a)$ for all $a\in C^{*}(\mathcal{G})$. For $a \in C^{*}(\mathcal{G})_{+}$ we can define
\begin{equation*}
\psi(a)=\int_{\mathcal{G}^{(0)}} \psi_{x}(a) \ d\mu(x)
\end{equation*}
Then $\psi$ is a non-zero weight on $C^{*}(\mathcal{G})$. Fatou's lemma implies that $\psi$ is lower semi-continuous and for $f\in C_{c}(\mathcal{G}^{(0)})_{+}$ then
\begin{equation*}
\lvert \psi(f) \rvert\leq \int_{\mathcal{G}^{(0)}} \lvert \psi_{x}(f) \rvert \ d\mu(x)
\leq \lVert f \rVert \mu(\supp(f)) < \infty \ .
\end{equation*}
In conclusion $\psi$ is a non-zero proper weight, and as in the proof of Proposition \ref{p41} one gets $C_{c}(\mathcal{G}) \subseteq \mathcal{M}_{\psi}$. For a sequence $\{V_{n}\}_{n=1}^{\infty}$ as introduced in the beginning, we can assume $0 < \mu(V_{n}) < \infty$ for all $n$. Setting $c_{n}:=c|_{\G_{n}}$ then $(\mu(V_{n})^{-1} \mu, \{\varphi_{x} \}_{x\in V_{n}})$ gives a $\beta$-KMS state $\omega_{n}$ on $C^{*}(\mathcal{G}_{n})$ for $\alpha^{c_{n}}$ by Theorem 1.3 in \cite{N}. For any $f\in C_{c}(\mathcal{G})$, there is a $n\in \mathbb{N}$ such that $f=\iota_{n}(f|_{\mathcal{G}_{n}})$, and we get
\begin{align*}
\psi(f^{*}f)&= \int_{\mathcal{G}^{(0)}} \sum_{g\in \mathcal{G}_{x}^{x}} (f^{*}f)(g) \varphi_{x}(u_{g}) \ d\mu(x) \\
&=\mu(V_{n})\int_{V_{n}} \sum_{g\in \mathcal{G}_{x}^{x}} (f^{*}f)(g) \varphi_{x}(u_{g}) \ d(\mu(V_{n})^{-1}\mu)(x) \\
&=\mu(V_{n}) \cdot \omega_{n}((f|_{\mathcal{G}_{n}})^{*} (f|_{\mathcal{G}_{n}}))
=\mu(V_{n}) \cdot \omega_{n}(\alpha^{c_{n}}_{-i\beta /2}(f|_{\mathcal{G}_{n}})
 \alpha^{c_{n}}_{-i\beta /2}(f|_{\mathcal{G}_{n}})^{*}) \\
&= \int_{V_{n}} \sum_{g\in \mathcal{G}_{x}^{x}} (\alpha^{c_{n}}_{-i\beta /2}(f|_{\mathcal{G}_{n}})
 \alpha^{c_{n}}_{-i\beta /2}(f|_{\mathcal{G}_{n}})^{*})(g) \varphi_{x}(u_{g}) \ d \mu(x) \\
&=\int_{V_{n}} \sum_{g\in \mathcal{G}_{x}^{x}} (\alpha^{c}_{-i\beta /2}(f)
 \alpha^{c}_{-i\beta /2}(f)^{*})(g) \varphi_{x}(u_{g}) \ d\mu(x) \\ 
 &=\psi(\alpha^{c}_{-i\beta /2}(f)
 \alpha^{c}_{-i\beta /2}(f)^{*}) \ .
\end{align*}
So we have proved that $\psi$ satisfies the $\beta$-KMS condition on
$C_{c}(\mathcal{G})$, yet to prove that it is a $\beta$-KMS weight we need to
prove this equality for all $a\in D(\alpha^{c}_{-i \beta
  /2})$. By \eqref{e1par} it follows that $\CcG$ consists of analytic elements for $\alpha^{c}$. Since $\alpha_{t}^{c}(\CcG)\subseteq \CcG$ for each $t\in \mathbb{R}$ then Corollary 1.22 in \cite{K2} implies that $\CcG$ is a core for $\alpha_{-i \beta /2}^{c}$. Hence for a $a \in D(\alpha_{-i \beta /2}^{c})$
we can find a sequence
$\{ f_{m} \}_{m\in \mathbb{N}} \subseteq C_{c}(\mathcal{G})$ such that
$f_{m} \to a$ and
$\alpha_{-i \beta /2}^{c}(f_{m}) \to \alpha_{-i \beta /2}^{c}(a)$ in
norm. For our approximate identity $\{E_{n}\}_{n\in \mathbb{N}}$ then
\begin{equation*}
\psi(E_{l}f_{m}^{*}E_{n}^{2}f_{m}E_{l})=\psi((E_{n}f_{m}E_{l})^{*}(E_{n}f_{m}E_{l}))=\psi(E_{n}\alpha^{c}_{-i\beta /2}(f_{m})E_{l}^{2}
 \alpha^{c}_{-i\beta /2}(f_{m})^{*}E_{n})
\end{equation*} 
for all $l,n \in \mathbb{N}$. Letting $m\to \infty$ we get
\begin{equation} \label{einfty}
\psi(E_{l}a^{*}E_{n}^{2}aE_{l})
=\psi(E_{n}\alpha^{c}_{-i\beta /2}(a)E_{l}^{2}
 \alpha^{c}_{-i\beta /2}(a)^{*}E_{n}) \ .
\end{equation}
Combining condition \eqref{eb3} and the definition of $\psi_{x}$ it follows that $\psi  \circ \alpha^{c}_{t} = \psi$ for each $t \in \mathbb{R}$, and since $\psi_{x}(E_{n}aE_{n})=\psi_{x}(a)E_{n}(x)^{2}$ we get by definition of $\psi$ that $\lim_{n\to \infty}\psi(E_{n}aE_{n})=\psi(a)$ for all $a\in C^{*}(\mathcal{G})_{+}$, so \eqref{einfty} and Lemma \ref{lhj} imply that $\psi$ is a $\beta$-KMS weight for $\alpha^{c}$.\newline

\paragraph{\textbf{Step $2$: The formula \eqref{eCcG} defines a unique $\beta$-KMS weight.}} \hfill

\noindent This follows if we can prove $\psi$ constructed in step $1$ is the
only $\beta$-KMS weight for $\alpha^{c}$ satisfying~\eqref{eCcG}. Let
$\psi'$ be a $\beta$-KMS weight for $\alpha^{c}$ that agrees with
$\psi$ on $C_{c}(\mathcal{G})$, and let
$\{ E_{n} \}_{n=1}^{\infty} \subseteq C_{c}(\mathcal{G}^{(0)})$ be the approximate identity defined in the beginning of the proof. Since both $\psi$ and $\psi'$ are finite on $E_{n+1}^{2}$, and since they agree on $E_{n}C_{c}(\mathcal{G})E_{n}$, $\psi$ and $\psi'$ agree on the sub-$C^{*}$-algebra $\overline{E_{n}C^{*}(\mathcal{G})E_{n}}$ for all $n\in \mathbb{N}$, and hence $\psi=\psi'$ by injectivity of the map in Proposition \ref{p45}. \newline

\paragraph{\textbf{Step $3$: Associating a pair $(\mu, \{\varphi_{x}\}_{x\in \mathcal{G}^{(0)}})$ to a weight $\psi$}} \hfill

\noindent Since $\CcG \subseteq \mathcal{M}_{\psi}$ by Proposition \ref{p41} then $\psi$ is a positive linear functional on $\CcGu$, so by the Riesz Representation Theorem there is a unique regular Borel measure $\mu$ on $\mathcal{G}^{(0)}$ such that
\begin{equation*}
\psi(f)=\int_{\mathcal{G}^{(0)}} f(x) \ d \mu(x) \quad \text{ for all } f\in C_{c}(\mathcal{G}^{(0)}) \ .
\end{equation*}
If $\mu=0$ then $\psi=0$ on $C_{c}(\mathcal{G})$, and by lower semi-continuity then $\psi=0$ and we are done. Assume therefore that $\mu \neq 0$, then for the sequence $\{V_{n}\}_{n=1}^{\infty}$ as introduced in the beginning of the proof we can assume that $0<\mu(V_{n})<\infty$ for all $n$. For each $n\in \mathbb{N}$ we define $\omega_{n}$ on $C^{*}(\mathcal{G}_{n})$ by
\begin{equation*}
\omega_{n}(a)=\mu(V_{n})^{-1} \psi(\iota_{n}(a)) \ .
\end{equation*}
Since $E_{n}\iota_{n}(a)E_{n}=\iota_{n}(a)$ for all $a\in C^{*}(\mathcal{G}_{n})$ we get that $\psi(\iota_{n}(a))=\psi(E_{n}\iota_{n}(a)E_{n})<\infty$ for all
$a\in C^{*}(\mathcal{G}_{n})_{+}$, proving that $\omega_{n}$ is a
positive linear functional, and by definition of $\mu$ then $\omega_{n}$ is a state. Since
$c_{n}=c|_{\mathcal{G}_{n}}$ is a continuous groupoid homomorphism on
$\mathcal{G}_{n}$ and
$\iota_{n} \circ\alpha^{c_{n}}=\alpha^{c}\circ \iota_{n}$, we get that
$\omega_{n}$ is a $\beta$-KMS state for $\alpha^{c_{n}}$ on
$C^{*}(\mathcal{G}_{n})$. Using Theorem~1.3 in \cite{N}, we get a
regular Borel probability measure $\mu_{n}$ on
$V_{n}=\mathcal{G}_{n}^{(0)}$ and a $\mu_{n}$-measurable field of
states $\{ \varphi_{x}^{n} \}_{x\in V_{n}}$ such that:
\begin{itemize}
\item[$a_{n}$)] $\mu_{n}$ is quasi-invariant on $\mathcal{G}_{n}$ with Radon-Nikodym cocycle $e^{-\beta c_{n}}$,
\item[$b_{n}$)] $\varphi^{n}_{x}(u_{g})=\varphi^{n}_{r(h)}(u_{hgh^{-1}})$ for $\mu_{n}$-a.e. $x\in V_{n}$ and all $g\in (\mathcal{G}_{n})^{x}_{x}=\mathcal{G}^{x}_{x}$, $h\in (\mathcal{G}_{n})_{x}$, 
\item[$c_{n}$)] $\varphi^{n}_{x}(u_{g})=0$ for $\mu_{n}$-a.e. $x\in V_{n}$ and all $g\in \mathcal{G}_{x}^{x} \setminus c_{n}^{-1}(0)$,
\end{itemize}
such that for $f\in C_{c}(\mathcal{G}_{n})$ we have
\begin{equation*}
\mu(V_{n})^{-1} \psi(\iota_{n}(f)) =\int_{V_{n}} \sum_{g\in \mathcal{G}_{x}^{x}} f(g) \varphi^{n}_{x}(u_{g}) \ d \mu_{n} (x) \ .
\end{equation*}
For every function $f\in C_{c}(V_{n})$ we have
\begin{equation*}
\mu(V_{n})^{-1} \int_{V_{n}} f(x) \ d\mu (x)
=\mu(V_{n})^{-1} \psi(\iota_{n}(f)) 
=\int_{V_{n}} f(x) \ d\mu_{n}(x) \ ,
\end{equation*}
so $\mu |_{V_{n}}=\mu(V_{n}) \mu_{n}$, which implies that $\mu_{n} = \mu(V_{n+1})\cdot \mu(V_{n})^{-1} \mu_{n+1} |_{V_{n}}$. Any
small bisection $W$ satisfies $W\subseteq \mathcal{G}_{n}$ for all
sufficiently large $n$, so by Proposition~\ref{p33} $\mu$ satisfies
\eqref{eb1} since $\mu_{n}$ satisfies $a_{n})$ for each $n$. We can extend any $f\in C_{c}(\mathcal{G}_{n})$ by zero to a function $f'\in  C_{c}(\mathcal{G}_{n+1})$ and then
\begin{align*}
  &\int_{V_{n}} \sum_{g\in \mathcal{G}_{x}^{x}} f(g)
  \varphi_{x}^{n} (u_{g}) \ d \mu_{n}(x) = \mu(V_{n})^{-1} \psi(
  \iota_{n}(f))
  \\
  &=\frac{\mu(V_{n+1})}{\mu(V_{n})} \int_{V_{n+1}} \sum_{g\in
    \mathcal{G}_{x}^{x}} f'(g) \varphi_{x}^{n+1} (u_{g}) \ d
  \mu_{n+1}(x)
  \\
  &=\int_{V_{n}} \sum_{g\in \mathcal{G}_{x}^{x}} f(g)
  \varphi_{x}^{n+1} (u_{g}) \ d \mu_{n}(x) \ .
\end{align*}
Since $\mu(N)=0$ for $N\subseteq V_{n}$ iff $\mu_{n}(N)=0$ then
$\{\varphi_{x}^{n+1}\}_{x\in V_{n}}$ satisfies $b_{n})$ and $c_{n})$,
so injectivity of the map in Theorem 1.3 in \cite{N} implies that
$\varphi_{x}^{n+1}=\varphi_{x}^{n}$ for $\mu_{n}$-almost all
$x\in V_{n}$. Let $N_{n} \subseteq V_{n}$ be a Borel set with
$\mu_{n}(N_{n})=0$ and $\varphi_{x}^{n+1}=\varphi_{x}^{n}$ for
$x\in V_{n}\setminus N_{n}$, and set $N=s(r^{-1}(\bigcup_{n}
N_{n}))$. Then $N$ is Borel and $\mu(N)=0$ by
Lemma~\ref{l34}. Let $Tr_{x}$ be the
canonical trace on $C^{*}(\mathcal{G}_{x}^{x})$, i.e. the trace with $\text{Tr}_{x}(u_{g})=0$ for $g\neq x$, and set
\begin{equation*}
\varphi_{x} =
\begin{cases}
\varphi_{x}^{n}  &  \text{if } x\in V_{n} \setminus N \text{ for some } n\ , \\
Tr_{x}  &  \text{if } x\in N \ .
\end{cases}
\end{equation*}
This is well defined by choice of $N$, and it is straightforward to
check that $\{\varphi_{x}\}_{x\in  \mathcal{G}^{(0)}}$ satisfies \eqref{eb2} and \eqref{eb3}. Combining that any $f\in C_{c}(\mathcal{G})$ satisfies
$\supp(f)\subseteq \mathcal{G}_{n}$ for large $n$ and that
$\{\varphi_{x}^{n}\}_{x\in V_{n}}$ is $\mu_{n}$-measurable it follows that $\{\varphi_{x}\}_{x\in \mathcal{G}^{(0)}}$ is $\mu$-measurable. For $f\in C_{c}(\mathcal{G})$ there exists a
$n$ such that $f=\iota_{n}(f |_{\mathcal{G}_{n}})$, and hence
\begin{align*}
\psi(f) &= \psi( \iota_{n}(f |_{\mathcal{G}_{n}}))
=\mu(V_{n}) \int_{V_{n}} \sum_{g\in \mathcal{G}_{x}^{x}} f(g) \varphi^{n}_{x}(u_{g}) \ d \mu_{n} (x)\\
&=\int_{\mathcal{G}^{(0)}} \sum_{g\in \mathcal{G}_{x}^{x}} f(g) \varphi_{x}(u_{g}) \ d \mu (x)
\end{align*}
which proves that to any $\beta$-KMS weight $\psi$ corresponds a pair $(\mu, \{\varphi_{x}\}_{x\in \mathcal{G}^{(0)}})$. \newline

\paragraph{\textbf{Step $4$: The map is a bijection.}} \hfill

\noindent By step $1$ and step $2$ the map is well defined, and by step $3$ it is surjective. So to prove the Theorem we only need to prove that the map is injective, but this follows exactly as in the last paragraph of the proof of Theorem 3.2 in \cite{CT5}.
\end{proof}

%Using Theorem \ref{tNW} we can now conclude that for diagonal actions on the $C^{*}$-algebra of \'etale groupoids the definition of KMS weights given by %Renault in Proposition II.5.4 in \cite{Re} and by Combes in \cite{Combes} agrees. 
%\begin{cor}
%The map $\psi\to \psi|_{\CcG}$ is an affine bijection between the $\beta$-KMS weights and the $\beta$-KMS positive type measures as defined in \cite{Re}. 
%\end{cor}

\begin{remark}
The bijection in Theorem \ref{tNW} restricts to a bijection between KMS states and the pairs $(\mu, \{\varphi_{x}\}_{x\in \mathcal{G}^{(0)}})$ where $\mu$ is a probability measure, which makes the analogy in the introduction strikingly accurate in our setting.
\end{remark}

\begin{remark} \label{rem}As observed in \cite{N}, condition~\eqref{eb3} in
  Theorem~\ref{tNW} is automatically satisfied when $\beta \neq 0$,
  because a quasi-invariant measure $\mu$ with Radon-Nikodym cocycle
  $e^{-\beta c}$ will automatically be concentrated on the
  $x\in \mathcal{G}^{(0)}$ with
  $\mathcal{G}_{x}^{x} \subseteq c^{-1}(0)$. To see this, set $M=\left \{g\in \mathcal{G} : c(g)>0 \text{ and } r(g)=s(g) \right \}$ and notice that
\begin{equation*}
s(M)=\{ x\in \mathcal{G}^{(0)} \ : \ \mathcal{G}_{x}^{x} \nsubseteq c^{-1}(0) \}  
\end{equation*}
is Borel, and that for all small bisections $W$ then
\begin{equation*}
\mu(s(M\cap W)) = \int_{r(M\cap W)} e^{\beta c (r_{W}^{-1}(x))} \ d\mu(x) .
\end{equation*}
Since $r(M\cap W)=s(M\cap W)$ then $\mu(s(M\cap W))=0$ when $\beta \neq 0$, proving that $\mu(s(M))=0$.
\end{remark}

Theorem \ref{tNW} provides answers to some open questions in the literature, which we will answer below in Corollary \ref{cor66} and Theorem \ref{nysaet}, but first we will spell out a technical remark in Corollary \ref{lcircP}. As in the proof of Theorem \ref{tNW} we let $\text{Tr}_{x}$ denote the canonical trace on $C^{*}(\mathcal{G}_{x}^{x}) $ for $x\in \mathcal{G}^{(0)}$.

\begin{cor} \label{lcircP}
Let $\mathcal{G}$ be an \'etale groupoid, let $c:\mathcal{G} \to \mathbb{R}$ be a continuous groupoid homomorphism and let $\beta \in \mathbb{R}$. If $\psi$ is a $\beta$-KMS weight on $C^{*}(\mathcal{G})$ for $\alpha^{c}$ given by $(\mu, \{\varphi_{x}\}_{x\in \mathcal{G}^{(0)}})$ then $\psi=\psi\circ P$ if and only if $\varphi_{x}=\text{Tr}_{x}$ for $\mu$-a.e. $x\in \mathcal{G}^{(0)}$.
\end{cor}

\begin{proof}
  Theorem~\ref{tNW} implies that if two $\beta$-KMS weights agree on
  $C_{c}(\mathcal{G})$ then they are equal, but if
  $\psi \circ P =\psi$ then $\psi$ agrees with the $\beta$-KMS weight given by
  $(\mu, \{\text{Tr}_{x}\}_{x\in
    \mathcal{G}^{(0)}})$ on $\CcG$, proving one direction. Assume now that
  $\psi$ is given by
  $(\mu, \{\text{Tr}_{x}\}_{x\in
    \mathcal{G}^{(0)}})$ and consider $a\in
  C^{*}(\mathcal{G})$. Clearly $\psi(f)=\psi(P(f))$ for all
  $f\in C_{c}(\mathcal{G})$. Let
  $\{E_{m}\}_{m=1}^{\infty} \subseteq C_{c}(\mathcal{G}^{(0)})_{+}$ be an approximate identity as in the proof of Theorem \ref{tNW}, and assume that
  $\{h_{n}\}_{n=1}^{\infty} \subseteq C_{c}(\mathcal{G})$ is a
  sequence with $h_{n} \to a$ in norm. Then
\begin{align*}
\psi(E_{m}a^{*}aE_{m})=\lim_{n} \psi(E_{m}h_{n}^{*}h_{n}E_{m})= \lim_{n} \psi(E_{m}P(h_{n}^{*}h_{n})E_{m})= \psi(E_{m}P(a^{*}a)E_{m})
\end{align*}
for each $m\in \mathbb{N}$. By definition of $\psi$ then
\begin{equation*}
\psi(a^{*}a)=\lim_{m} \psi(E_{m}a^{*}aE_{m})=\lim_{m} \psi(E_{m}P(a^{*}a)E_{m})=\psi(P(a^{*}a))
\end{equation*}
proving the corollary.
\end{proof}

Combining Corollary \ref{lcircP}, Remark \ref{rem} and Theorem \ref{tNW} we get the following Corollary, which gives an affirmative answer to the question raised after Corollary 2.3 in \cite{T3}.

\begin{cor} \label{cor66}
Let $\G$ be an \'etale groupoid, let $c:\G\to\mathbb{R}$ be a continuous groupoid homomorphism with $\text{Ker}(c)\cap \G_{x}^{x}=\{x\}$ for all $x\in \G^{(0)}$ and let $\beta \neq 0$.

The following three sets are in a bijective correspondence:
\begin{enumerate}
\item \label{fjol1}The quasi-invariant measures on $\mathcal{G}^{(0)}$ with Radon-Nikodym cocycle $e^{-\beta c}$,
\item \label{fjol2}the $\beta$-KMS weights for $\alpha^{c}$ on $C^{*}(\G)$, and
\item \label{fjol3} the $\beta$-KMS weights for $\alpha^{c}$ on $C_{r}^{*}(\G)$.
\end{enumerate}
\end{cor}

\begin{proof}
The bijection between \eqref{fjol1} and \eqref{fjol2} is a consequence of Remark \ref{rem} and Theorem \ref{tNW}. If $\pi: C^{*}(\G) \to C_{r}^{*}(\G)$ denotes the canonical surjective $*$-homomorphism, then the map $\psi\to \psi \circ\pi$ is injective from the set \eqref{fjol3} to \eqref{fjol2}. Since any $\beta$-KMS weight $\psi$ on $C^{*}(\G)$ is given by $\psi(a)=\int_{\G^{(0)}} P(a)\ \mathrm{d}\mu$ for $a\in C^{*}(\G)_{+}$ for a quasi-invariant measure $\mu$ with Radon-Nikodym cocycle $e^{-\beta c}$ by Corollary \ref{lcircP}, surjectivity of the map follows from Proposition 2.1 in \cite{T3}.
\end{proof}

In \cite{CT5} Neshveyev's Theorem was proved for  \'etale groupoids $\mathcal{G}$ with $\mathcal{G}^{(0)}$ totally disconnected as part of the proof of Theorem 2.1 in \cite{CT5}. The assumption that $\mathcal{G}^{(0)}$ is totally disconnected in Theorem 2.1 in \cite{CT5} was mainly needed to ensure that Neshveyev's Theorem could be used, and hence using Theorem \ref{tNW} it follows that Theorem 2.1 in \cite{CT5} with condition $4)$ removed is valid also for groupoids where $\mathcal{G}^{(0)}$ is not totally disconnected. In conclusion we get Theorem \ref{nysaet} below by combining the results of \cite{CT5} with Theorem \ref{tNW}. For the statement of this theorem, notice that a \emph{diagonal} KMS weight $\psi$ is a KMS weight satisfying $\psi=\psi \circ P$. 

\begin{thm} [Theorem 2.1 in \cite{CT5}]\label{nysaet}
Let $\mathcal{G}$ be an \'etale groupoid such that for at least one element $x\in \mathcal{G}^{(0)}$ the isotropy group $\mathcal{G}_{x}^{x}$ is trivial, i.e. $\mathcal{G}_{x}^{x}=\{x\}$, and that $\mathcal{G}$ is minimal in the sense that $s(r^{-1}(\{y\}))$ is dense in $\mathcal{G}^{(0)}$ for all $y\in \mathcal{G}^{(0)}$.

Let $\alpha= \{\alpha_{t}\}_{t\in \mathbb{R}}$ be a continuous $1$-parameter group of automorphisms on $C_{r}^{*}(\mathcal{G})$ and assume for some $\beta_{0}\neq 0$ there is a non-zero $\beta_{0}$-KMS weight for $\alpha$. TFAE
\begin{enumerate}
\item There is a $\beta_{1}\neq 0$ and a non-zero diagonal $\beta_{1}$-KMS weight for $\alpha$.
\item Whenever $\beta\neq 0$ and there is a non-zero $\beta$-KMS weight for $\alpha$, there is also a non-zero diagonal $\beta$-KMS weight for $\alpha$.
\item $\alpha_{t}(f)=f$ for all $t\in \mathbb{R}$ and all $f\in C_{0}(\mathcal{G}^{(0)})$.
\item $\alpha$ is diagonal.
\end{enumerate}
\end{thm}

\subsection{The set of KMS weights on \'etale groupoid $C^{*}$-algebras}
Using Theorem \ref{tNW} we can elaborate the analysis in Section \ref{section5} in the case where the $C^{*}$-dynamical system is on the form $(C^{*}(\G), \alpha^{c})$. We will bring together the ideas from Section \ref{section5} and Section \ref{section3} and investigate the structure of the set of quasi-invariant measures.

Fix a sequence $\{E_{n}\}_{n=1}^{\infty}\subseteq C_{c}(\mathcal{G}^{(0)})_{+}$ as in the beginning of Section \ref{section5}, and keep the notation from Section \ref{section5}, i.e.  $\mathcal{A}_{n}:=\overline{E_{n}C^{*}(\mathcal{G})E_{n}}$ and we identify $\mathcal{W}(\beta, \alpha^{c})$ with $\varprojlim_{n\in \mathbb{N}} \mathcal{W}^{\beta}_{n}$. The set $\Delta(e^{-\beta c})$ of quasi-invariant regular measures with Radon-Nikodym cocycle $e^{-\beta c}$ can naturally be embedded into $\mathcal{W}(\beta, \alpha^{c})$ by mapping a measure $\mu$ onto the weight
\begin{equation*}
\psi_{\mu}(a)=\int_{\Gu} P(a) \ d \mu \quad \text{ for } a\in C^{*}(\G)_{+},
\end{equation*}
see e.g. Corollary \ref{lcircP}. We then get

\begin{lemma}
Let $\mathcal{G}$ be an \'etale groupoid, let $c:\mathcal{G} \to \mathbb{R}$ be a continuous groupoid homomorphism and let $\beta \in \mathbb{R}$. Then $\Delta (e^{-\beta c})$ is a closed convex cone in $\varprojlim_{n\in \mathbb{N}} \mathcal{W}^{\beta}_{n}$ and $\Delta (e^{-\beta c})$ is a lattice in its natural order.
\end{lemma}

\begin{proof}
Since $P(E_{n}xE_{n})=E_{n}P(x)E_{n}$ for all $n\in \mathbb{N}$ and $x\in C^{*}(\G)$ then $P:\mathcal{A}_{n} \to \mathcal{A}_{n}$, and if $\psi\in \varprojlim_{n\in \mathbb{N}} \mathcal{W}^{\beta}_{n}$ we get that $\psi=\psi\circ P$ if and only if $\psi\circ \iota_{n, \infty} \circ P = \psi\circ \iota_{n, \infty}$ for all $n\in \mathbb{N}$ as in Corollary \ref{lcircP}. Since the $P$-invariant elements of $\mathcal{W}_{n}^{\beta}$ are closed, this proves that $\Delta (e^{-\beta c})$ is closed. Since $\Delta (e^{-\beta c})$ is clearly a convex cone this proves the first assertion. Assume now that $\psi, \phi \in \Delta (e^{-\beta c})$ and let $\omega$ be their least upper bound in $\mathcal{W}(\beta, \alpha^{c}) $. There exists a $\rho\in \mathcal{W}(\beta, \alpha^{c})$ such that $\phi+\rho=\omega$, and hence by Theorem \ref{tNW} $\rho \circ P$ and $\omega \circ P$ are also $\beta$-KMS weights. For any $n\in \mathbb{N}$ and $x\in \mathcal{A}_{n}$ then
\begin{equation*}
\phi(x )+\rho \circ P( x )= (\phi+\rho)(P( x ))=\omega \circ P(x ) \ ,
\end{equation*}
so $\phi \leq \omega \circ P $ by the observation in \eqref{eobsny}, and likewise $\psi \leq \omega \circ P$. Since $\omega$ is the least upper bound this implies that $\omega \leq \omega \circ P$, but $\omega \circ P(E_{n})-\omega(E_{n})=0$ for all $n$, implying that $\omega \circ P = \omega$. Hence $\omega \in \Delta(e^{-\beta c})$. A similar argument gives that the greatest lower bound lies in $\Delta(e^{-\beta c})$, proving the Lemma. 
\end{proof}

The set of KMS states on a unital $C^{*}$-algebra is a simplex, and the fact that this allows for unique maximal barycentric decompositions is often an essential tool when working with KMS states. As a last remark, we will therefore combine our results so far to obtain a similar description of KMS weights for a large class of groupoids containing in particular the minimal \'etale groupoids.

\begin{prop} \label{pdeltasimplex}
Let $\mathcal{G}$ be an \'etale groupoid, let $c : \mathcal{G} \to \mathbb{R}$ be a continuous groupoid homomorphism and let $\beta \in \mathbb{R}$. If $f\in C_{c}(\mathcal{G}^{(0)})_{+}$ satisfies 
\begin{equation*}
s(r^{-1}(f^{-1}(]1, \infty[)))=\mathcal{G}^{(0)}
\end{equation*}
then $\psi(f)\in ]0, \infty[$ for each $\psi \in \mathcal{W}(\beta, \alpha^{c})\setminus \{0 \}$ and the set
\begin{equation*}
\mathcal{W}_{f}(\beta, \alpha^{c}):=\{ \psi\in \mathcal{W}(\beta, \alpha^{c}) |\ \psi(f)=1\} 
\end{equation*}
is a simplex in $\varprojlim_{n\in \mathbb{N}} \mathcal{W}^{\beta}_{n}$.
\end{prop}

\begin{proof}
Let $\psi \in \mathcal{W}(\beta, \alpha^{c})$. If $ \psi(f)=0$ and $\psi$ on $\Gu$ is given by the measure $\mu \in \Delta(e^{-\beta c})$ then $\mu (f^{-1}(]1, \infty[))=0$, and hence $\mu(\mathcal{G}^{(0)})=0$ by Lemma \ref{l34}. Hence any $\psi \in \mathcal{W}(\beta, \alpha^{c})\setminus \{0\}$ can be written $\psi =\lambda \psi'$ for a unique $\lambda>0$ and $\psi' \in \mathcal{W}_{f}(\beta, \alpha^{c})$. It follows from this that $\mathcal{W}_{f}(\beta, \alpha^{c})$ is a simplex if it is a compact subset of $\varprojlim_{n\in \mathbb{N}} \mathcal{W}^{\beta}_{n}$, c.f. page 334 in \cite{BR}. Since $\mathcal{W}_{f}(\beta, \alpha^{c})$ is clearly closed in $\varprojlim_{n\in \mathbb{N}} \mathcal{W}^{\beta}_{n}$ it suffices to prove that it is a subset of a compact set.

For fixed $n\in \mathbb{N}$ there exist small bisections $W_{1}, \dots , W_{k}$ in $\mathcal{G}$ with $\text{supp}(E_{n}) \subseteq \bigcup_{l=1}^{k} s(W_{l})$ and $r(W_{l})\subseteq f^{-1}(]1,\infty[)$ for all $l$. Let $\psi \in \mathcal{W}_{f}(\beta, \alpha^{c})$ and let $\mu \in \Delta(e^{-\beta c})$ be the measure associated to it by Theorem \ref{tNW}, then
\begin{equation*}
\mu(s(W_{l})) = \int_{r(W_{l})} e^{\beta c (r_{W_{l}}^{-1}(x))} \ d \mu (x) \leq \sup_{g \in \overline{W_{l}}} e^{\beta c (g)} \mu(r(W_{l})) \ .
\end{equation*}
Since $\psi(f)=1$ and $r(W_{l})\subseteq f^{-1}(]1, \infty[)$ then $\mu(r(W_{l})) \leq 1$. It follows that for any $\psi \in \mathcal{W}_{f}(\beta, \alpha^{c})$ then
\begin{equation*}
\psi(E_{n}^{2})\leq \sum_{l=1}^{k}\sup_{g \in \overline{W_{l}}} e^{\beta c (g)} \ ,
\end{equation*}
and hence $\mathcal{W}_{f}(\beta, \alpha^{c})$ is contained in a compact set by Tychonoffs Theorem.
\end{proof}

\section{A refinement of Neshveyev's Theorem} \label{section6}
The $\mu$-measurable fields of states occuring in Neshveyev's Theorem are in general difficult to describe, but in Theorem 5.2 in \cite{C2} we gave a description of these $\mu$-measurable fields of states for a large class of groupoids. This description essentially boils the analysis of the KMS states down to the analysis of the quasi-invariant measures. The purpose of this section is to do the same for KMS weights, c.f. Corollary \ref{cortw} and Theorem \ref{gns} below.

In this section we will generalise Theorem 5.2 in \cite{C2} to KMS weights. The proof of Theorem 5.2 in \cite{C2} is quite technical and long, so instead of trying to adapt it to the setting of KMS weights, we present a new proof which is much simpler and which also works elegantly for KMS weights. The main idea in our new proof is to use ideas from ergodic theory to control the behaviour of the $\mu$-measurable fields of states.

We will throughout this section restrict attention to the following groupoids.

\begin{defn}\label{dabel}
Let $A$ be a discrete countable abelian group and let $\mathcal{G}$ be an \'etale groupoid. We say that $\mathcal{G}$ is \emph{injectively graded by $A$} if there is a continuous groupoid homomorphism $\Phi: \mathcal{G} \to A $ satisfying
\begin{equation} \label{einjectiv}
\ker(\Phi) \cap \mathcal{G}_{x}^{x} =\{x\} \qquad \text{for all } x\in \mathcal{G}^{(0)} \ .
\end{equation}
\end{defn}

\begin{remark}
The criterion in \eqref{einjectiv} is equivalent with $\Phi$ being injective on all isotropy groups. There are several important classes of $C^{*}$-algebras that can be realized as groupoid $C^{*}$-algebras for groupoids satisfying Definition \ref{dabel}. As an example, the groupoids arising from compactly aligned topological $k$-graphs as described in \cite{Yeend} satisfies Definition \ref{dabel}, see e.g. Example 2.3 in \cite{C2} for an explanation of this. This implies in particular that the groupoids that give rise to directed graph $C^{*}$-algebras, higher rank graph $C^{*}$-algebras and crossed products by $\mathbb{Z}^{k}$ satisfy Definition \ref{dabel}, see Example 7.1 in \cite{Yeend} for the details.

In Definition 2.1 in \cite{C2} the groupoids in Definition \ref{dabel} with compact unit space were introduced under a different name, but in the meantime the author has become acquainted with the better suited notion of graded groupoids.
 \end{remark}

\begin{thm} \label{tmain}
Let $\mathcal{G}$ be an \'etale groupoid injectively graded by a discrete countable abelian group $A$ via a map $\Phi: \mathcal{G} \to A$, let $c:\mathcal{G}\to \mathbb{R}$ be a continuous groupoid homomorphism and let $\beta \in \mathbb{R}$. If $\mu \in \Delta(e^{-\beta c})\setminus\{0\}$ is ergodic then:
\begin{enumerate}
\item \label{a1}The subset
\begin{equation*}
X(C):=\{x\in \mathcal{G}^{(0)} \ : \ \Phi(\mathcal{G}_{x}^{x}) = C \}
\end{equation*}
is Borel and invariant for each subgroup $C\subseteq A$.
\item \label{a2} There exists a unique subgroup $B$ of $A$ with $\mu(X(B)^{C})=0$.
\item\label{a3} For $x\in X(B)$ let $\Phi_{x}: C^{*}(\mathcal{G}_{x}^{x}) \to C^{*}(B)$ be the isomorphism induced by the restriction of $\Phi$. If $\{\varphi_{x}\}_{x \in \mathcal{G}^{(0)}}$ is a $\mu$-measurable field of states with
\begin{equation} \label{esikkert}
\varphi_{x}(u_{g}) =\varphi_{r(h)}(u_{hgh^{-1}}) \text{ for }\mu \text{-a.e. } x\in \mathcal{G}^{(0)} \text{ and all } g\in \mathcal{G}_{x}^{x}\text{ and } h\in \mathcal{G}_{x} 
\end{equation}
then there exists a state $\varphi$ on $C^{*}(B)$ such that $\varphi \circ \Phi_{x} = \varphi_{x}$ for $\mu$-a.e. $x\in X(B)$.
\end{enumerate}
\end{thm}

\begin{proof}
The proof of \eqref{a1} and \eqref{a2} follows as in the proof of $1)$ in Theorem 5.2 in \cite{C2} and we therefore leave the verification to the reader. To prove \eqref{a3} notice that since we can realise the Borel function $1_{\Phi^{-1}(\{a\})}$ as the
point wise limit of functions in $C_{c}(\mathcal{G})$ we can use Lemma \ref{l34} to find an invariant Borel $\mu$-null set $N\subseteq \mathcal{G}^{(0)}$ such that 
\begin{equation*}
\mathcal{G}^{(0)}\setminus N \ni x \to \sum_{g\in \mathcal{G}_{x}^{x}} 1_{\Phi^{-1}(\{a\})} (g) \varphi_{x}(u_{g})
\end{equation*}
is Borel for each $a\in A$ and such that $\varphi_{x}(u_{g}) =\varphi_{r(h)}(u_{hgh^{-1}})$ for all
$x\in \mathcal{G}^{(0)}\setminus N$ and all
$g\in \mathcal{G}_{x}^{x}$ and $ h\in \mathcal{G}_{x}$. Set
$\varphi'_{x}=\varphi_{x}$ for $x\notin N$ and
$\varphi'_{x}=\text{Tr}_{x}$ for $x\in N$,
where $\text{Tr}_{x}$ denotes the canonical
trace on $C^{*}(\mathcal{G}_{x}^{x})$. The map 
\begin{equation}\label{eborel}
\mathcal{G}^{(0)} \ni x \to \sum_{g\in \mathcal{G}^{x}_{x}} 1_{\Phi^{-1}(\{a\})}(g) \varphi'_{x}(u_{g})
\end{equation}
is Borel for each $a\in A$ and the equality in \eqref{esikkert} is true for $\{\varphi_{x}'\}_{x\in \mathcal{G}^{(0)}}$ for all $x\in \mathcal{G}^{(0)}$. Let $E$ denote the weak$^{*}$ compact set of states on $C^{*}(B)$, and notice that the sets
\begin{equation*}
\{\varphi \in E \ : \ \lvert \varphi(u_{b}) -\omega(u_{b}) \rvert <  \varepsilon \} \quad \text{for } \omega \in E, \ \varepsilon >0 \text{ and } b\in B
\end{equation*}
are a subbasis for the weak$^{*}$ topology on $E$. For each $b\in B$ then
\begin{equation*}
X(B) \ni x \to \sum_{g\in \mathcal{G}_{x}^{x}} 1_{\Phi^{-1}(\{b\})} (g) \varphi_{x}'(u_{g})
=\varphi_{x}'(\Phi_{x}^{-1}(u_{b})) \ ,
\end{equation*}
and hence it follows from \eqref{eborel} and the definition of our subbasis that the map $X(B)\ni x \to \varphi_{x}' \circ \Phi_{x}^{-1}\in E$ is Borel. Since $B$ is abelian the equality $\varphi_{x}'(u_{g}) =\varphi_{r(h)}'(u_{hgh^{-1}})$ for $g\in \mathcal{G}_{x}^{x}$ and $h\in \mathcal{G}_{x}$ implies that the map is constant on $s(r^{-1}(\{x\}))$ for each $x\in X(B)$, and hence it is constant $\mu$-almost everywhere by Remark \ref{lergokonstant}. This implies that there is a state $\phi$ on $C^{*}(B)$ with $\phi\circ \Phi_{x}=\varphi_{x}'$ for almost all $x$, and hence the same is true for $\{\varphi_{x}\}_{x\in \mathcal{G}^{(0)}}$.
\end{proof}

With Theorem~\ref{tmain} we can describe the KMS weights on the
groupoid $C^{*}$-algebra $C^{*}(\mathcal{G})$ of a groupoid $\mathcal{G}$ injectively graded by a
discrete countable abelian group.

\begin{thm} \label{tweights} Let $\mathcal{G}$ be an \'etale groupoid injectively graded by a discrete countable abelian group $A$ via a map $\Phi: \mathcal{G} \to A$. Let $\beta \in \mathbb{R}\setminus \{0 \}$, let  $c: \mathcal{G} \to \mathbb{R}$ be a continuous groupoid homomorphism and let
  $\mu\in \Delta(e^{-\beta c}) \setminus \{0\}$ be ergodic. Denote by $B$
  the subgroup of $A$ associated to $\mu$ given by \eqref{a2} in
  Theorem~\ref{tmain}. There exists an affine bijection from the
  state-space on $C^{*}(B)$ to the $\beta$-KMS weights on $C^{*}(\mathcal{G})$ that restricts to $\mu$ on $C_{c}(\mathcal{G}^{(0)})$. A state $\varphi$ maps to the $\beta$-KMS weight $\omega_{\varphi}$ given by
\begin{equation} \label{eweights}
\omega_{\varphi}(f)=\int_{X(B)} \sum_{g\in \mathcal{G}_{x}^{x}} f(g) \varphi(u_{\Phi(g)}) \ d\mu (x) \qquad \text{ for all } f\in C_{c}(\mathcal{G}) \ .
\end{equation}
\end{thm}

\begin{proof}
If $\varphi$ is a state on $C^{*}(B)$ we define a field of states by
\begin{equation*}
\varphi_{x}=
\begin{cases}
\varphi \circ \Phi_{x} & \text{ for } x\in X(B) \ , \\
Tr_{x} & \text{ for } x\notin X(B).
\end{cases}
\end{equation*}
Since any $f\in C_{c}(\mathcal{G})$ is a sum of functions supported on sets $\Phi^{-1}(\{a\})$ with $a\in A$ it is straightforward to check that this is a $\mu$-measurable field of states satisfying the conditions in
Theorem~\ref{tNW}. This way we get a map that is surjective by~\eqref{a3} in Theorem~\ref{tmain}. For injectivity notice that if
$\omega_{\varphi}=\omega_{\psi}$ for some states $\varphi, \psi$ on
$C^{*}(B)$ then $\varphi \circ \Phi_{x}=\psi\circ \Phi_{x}$ for
$\mu$-almost all $x \in X(B)$, so since $\Phi_{x}$ is an isomorphism
this implies that $\varphi=\psi$. Combining \eqref{eweights} and the uniqueness statement in Theorem \ref{tNW} we see that  $\omega_{t\varphi+(1-t)\phi}=t\omega_{\varphi}+(1-t)\omega_{\phi}$, and hence the map is affine.
\end{proof}

 When $B$ is a subgroup of a discrete abelian group $A$ as in Theorem \ref{tweights} then
  $C^{*}(B)\simeq C(\widehat{B})$, so the state-space of $C^{*}(B)$ is
  homeomorphic to the space of regular Borel probability measures on
  the Pontryagin dual $\widehat{B}$ of $B$. Notice that
  Theorem~\ref{tweights} also gives a description of the proper
  tracial weights on $C^{*}(\mathcal{G})$ by taking the groupoid
  homomorphism $c$ to be the zero function and $\beta \neq 0$. Using the above two theorems we can now give a very elegant description of the extremal KMS weights.
  
  \begin{cor}\label{cortw}
  In the setting of Theorem~\ref{tweights} there is a bijection
  between non-zero extremal $\beta$-KMS weights for $\alpha^{c}$ and pairs $(\mu, \xi)$ consisting of an ergodic
  measure $\mu\in \Delta(e^{-\beta c})\setminus\{0\}$ and a character
  $\xi\in \widehat{B}$ where $B\subseteq A$ is the subgroup
  corresponding to $\mu$ via Theorem~\ref{tmain}.
\end{cor}

\begin{proof}
A $\beta$-KMS weight $\psi$ with associated measure $\mu\in \Delta(e^{-\beta c})$ is extremal if and only if $\mu$ is extremal in $\Delta(e^{-\beta c})$ and $\psi$ is extremal in the convex set of $\beta$-KMS weights that restrict to $\mu$ on $C_{c}(\mathcal{G}^{(0)})$. The Corollary therefore follows from Theorem \ref{t35}.
\end{proof}

\subsection{The GNS representations of extremal KMS weights}
It follows from Corollary \ref{cortw} that to describe the extremal KMS weights for diagonal actions on \'etale groupoid $C^{*}$-algebras arising from groupoids injectively graded by abelian groups, it suffices to describe the quasi-invariant measures and their support. It turns out, that in this setting, the factor type of the extremal KMS weights only depends on the quasi-invariant measure associated to the KMS weight.

\begin{thm}\label{gns}
Let $\mathcal{G}$ be an \'etale groupoid injectively graded by a discrete countable abelian group $A$ via a map $\Phi: \mathcal{G} \to A$. Let $\beta \in \mathbb{R}\setminus \{0 \}$, let  $c: \mathcal{G} \to \mathbb{R}$ be a continuous groupoid homomorphism and let
  $\mu\in \Delta(e^{-\beta c}) \setminus \{0\}$ be ergodic. 

The von Neumann algebras generated by the GNS representations of the extremal KMS weights that restricts to $\mu$ on $C_{c}(\mathcal{G}^{(0)})$ are all isomorphic.
\end{thm}

\begin{proof}
It follows from Corollary \ref{cortw} that when $B\subseteq A$ is the abelian group associated to the ergodic measure $\mu\in \Delta(e^{-\beta c}) \setminus \{0\}$ then all extremal $\beta$-KMS weights for $\alpha^{c}$ that restricts to $\mu$ on $C_{c}(\mathcal{G}^{(0)})$ are given as pairs $(\mu, \xi)$ with $\xi\in \widehat{B}$. Let $\xi_{1}, \xi_{2} \in \widehat{B}$ and denote by $\psi_{1}$ and $\psi_{2}$ the $\beta$-KMS weights corresponding to respectively $(\mu, \xi_{1})$ and $(\mu, \xi_{2})$. By possibly extending the character, c.f. Theorem 2.1.4 in \cite{Rudin}, we can assume that $\xi_{1}\xi_{2}^{-1}\in \widehat{A}$. Letting $\gamma$ denote the action defined via $\Phi$ in \eqref{e1par}, then for any $f\in C_{c}(\mathcal{G})$ we have
\begin{equation*}
\psi_{2}(\gamma_{\xi_{1}\xi_{2}^{-1}}(f))=\int_{X(B)} \sum_{g\in \mathcal{G}_{x}^{x}}  (\xi_{1}\xi_{2}^{-1})(\Phi(g))f(g)\xi_{2}(\Phi(g))\ d\mu(x) =\psi_{1}(f).
\end{equation*}
Since $\gamma$ and $\alpha^{c}$ commute it is straightforward to check that $\psi_{2}\circ \gamma_{\xi_{1}\xi_{2}^{-1}}$ is a $\beta$-KMS weight for $\alpha^{c}$, and hence $\psi_{2}\circ \gamma_{\xi_{1}\xi_{2}^{-1}}=\psi_{1}$ by the uniqueness statement in Theorem \ref{tNW}. If $(H, \pi, \Lambda)$ denotes the GNS-triple for $\psi_{2}$ this implies that $(H, \pi\circ \gamma_{\xi_{1}\xi_{2}^{-1}}, \Lambda\circ \gamma_{\xi_{1}\xi_{2}^{-1}})$ is a GNS triple for $\psi_{1}$, proving the theorem.

\end{proof}

\end{document}